\documentclass[a4paper,11pt, reqno, oneside]{amsart}

\usepackage[utf8]{inputenc}
\usepackage{amsmath, amsfonts, amssymb, amsthm, mathtools, bbm, enumerate}
	
\usepackage{hyperref}

\usepackage[all]{xy}
\usepackage[height=23.0cm, width=15.5cm]{geometry}
\usepackage[active]{srcltx}

\theoremstyle{plain}
\newtheorem{thm}{Theorem}[section]
\newtheorem{lemma}[thm]{Lemma}
\newtheorem{prop}[thm]{Proposition}
\newtheorem{cor}[thm]{Corollary}

\theoremstyle{definition}
\newtheorem{defi}[thm]{Definition}
\newtheorem{notation}[thm]{Notation}
\newtheorem*{convention}{Convention}
\newtheorem*{question}{Question}
\newtheorem{rmk}[thm]{Remark}
\newtheorem{ex}[thm]{Example}
\newtheorem*{ex*}{Example}

\begin{document}

\newgeometry{height=23.8cm, width=15.6cm}
\title[On the Lipman-Zariski conjecture for logarithmic vector fields]{On the Lipman-Zariski conjecture for logarithmic vector fields on log canonical pairs}
\author{Hannah Bergner}

\address{Mathematisches Institut, Albert-Ludwigs-Universität Freiburg, Eckerstraße 1, 79104
Freiburg im Breisgau, Germany}
\email{Hannah.Bergner@math.uni-freiburg.de}

\thanks{Financial support by the DFG-Graduiertenkolleg  GK1821 "Cohomological Methods in Geometry" at the University of Freiburg
is gratefully acknowledged.}
\date{\today}

\begin{abstract}
We consider a version of the Lipman-Zariski conjecture for logarithmic vector fields and logarithmic $1$-forms on pairs.
Let $(X,D)$ be a pair consisting of a normal complex variety $X$ and an effective Weil divisor~$D$ such that the sheaf of logarithmic vector fields (or dually the sheaf of reflexive logarithmic $1$-forms) is locally free. 
We prove that in this case the following holds: 
If $(X,D)$ is dlt, then $X$ is necessarily smooth and $\lfloor D\rfloor $ is snc.
If $(X,D)$ is lc or the logarithmic $1$-forms are locally generated by closed forms, then $(X,\lfloor D\rfloor)$ is toroidal.
\end{abstract}

\maketitle

\tableofcontents

\section{Introduction}
The Lipman-Zariski conjecture posed in \cite[p.\ 874]{Lipman} states that every normal complex space with locally free tangent sheaf is smooth.
In this paper, we consider a version of this conjecture for the logarithmic tangent sheaf 
$\mathcal T_X(-\log D )$, or equivalently the dual sheaf~$\Omega_X^{[1]}(\log D)$ of reflexive logarithmic $1$-forms, on a pair $(X,D)$,
where $X$ is a normal complex quasi-projective variety and $D$ a reduced Weil divisor;
for precise definitions see Section~\ref{section: notation}.

\begin{ex}[Snc pair]
Let $X=\mathbb A^n$ with coordinates $z_1,\ldots, z_n$ and 
\mbox{$D=\{z_1\cdot\ldots\cdot z_k=0\}$.}
Then $z_1\frac{\partial}{\partial z_1},\ldots, z_k\frac{\partial}{\partial z_k}, \frac{\partial}{\partial z_{k+1}}, \ldots , \frac{\partial}{\partial z_n}$ form an 
$\mathcal O_X$-basis of the \mbox{logarithmic} tangent sheaf , and the dual sheaf of logarithmic $1$-forms is spanned by $\frac{dz_1}{z_1},\ldots, \frac{dz_k}{z_k}, dz_{k+1},\ldots, dz_n$.

More generally, if $X$ is smooth and $D$ is a reduced snc divisor,
then the logarithmic tangent sheaf $\mathcal T_X(-\log D)$ and its dual are locally free.
\end{ex}
 
\begin{ex}[Toric variety]\label{ex: toric var}
If $(X,D)$ is a pair consisting of a toric variety $X$ and a reduced divisor $D$ whose support is the complement of the open torus orbit,
then the sheaf $ \Omega_X^{[1]}(\log D)$ of reflexive logarithmic $1$-forms 
is free.
\end{ex}

This raises the question in which cases a converse of this is locally true: 
\begin{question}
Let $(X,D)$ be a pair such that the logarithmic tangent sheaf 
$\mathcal T_X(-\log D)$, or equivalently $\Omega_X^{[1]}(\log D)$, is locally free.
Is $(X,D)$ then necessarily toroidal, i.e.\ locally of the form as in Example~\ref{ex: toric var}?
\end{question}

\restoregeometry

In general, this is false. Consider for instance the following example:
\begin{ex}
Let $X=\mathbb A_\mathbb C^2$ and $D=\{y^2=x^3\}$. Then $\Omega_X^{[1]}(\log D)$ is locally free and $D$ is irreducible, but $D$ is not normal.
\end{ex}

For smooth varieties $X$ and arbitrary reduced divisors $D$ logarithmic vector fields, logarithmic differential forms and their properties have been studied a lot. 
Recently, precise conditions under which a reduced divisor $D$ in a smooth variety $X$ such that $\mathcal T_X(-\log D)$ is locally free is normal crossing were given in \cite{GrangerSchulze} and \cite{Faber}. 

In this article,
we study the pairs $(X,D)$ with locally free sheaf $\Omega_X^{[1]}(\log D )$,
where $X$ is allowed to be singular.
If $D=\sum_i a_iD_i$, $a_i\in \mathbb Q$,
is an effective Weil divisor, let  
$\lfloor D\rfloor =\sum_i\lfloor a_i\rfloor D_i$ denote its rounddown.
We completely answer the above question for pairs $(X,\lfloor D\rfloor )$ such that there is a pair $(X,D)$ that is dlt or lc, or such that the sheaf $\Omega_X^{[1]}(\log \lfloor D\rfloor )$ of reflexive logarithmic $1$-forms is locally generated by closed forms:

\begin{thm}[cf.\ Theorems~\ref{thm: dlt pairs}, \ref{thm: closed one forms}, \ref{thm: locally free lc pair}]\label{thm: main thm}
Let $(X,D)$ be a pair
consisting of a normal quasi-projective variety $X$ and a divisor $D=\sum_i a_i D_i$, where $D_i$ are distinct prime divisors, $a_i\in \mathbb Q$ and $0\leq a_i\leq 1$.
Assume that $\Omega_X^{[1]}(\log \lfloor D\rfloor)$ is locally free. Then the following holds:
\begin{enumerate}
\item[(a)] If the sheaf $\Omega_X^{[1]}(\log \lfloor D\rfloor)$
of reflexive logarithmic $1$-forms is locally generated by closed forms, then $(X,\lfloor D\rfloor)$ is toroidal.
\item[(b)] If the pair $(X,D)$ is dlt, then $X$ is smooth and
$\lfloor D\rfloor$ is an snc divisor.
If $(X,D)$ is lc, then $(X,\lfloor D\rfloor)$ is toroidal.
\end{enumerate}
\end{thm}

Recall that we call a pair $(X,D)$ toroidal if $X$ is locally (in the analytic topology) isomorphic to a toric variety $Y$ and $D$
is a reduced divisor corresponding to the complement of the open torus orbit in $Y$.
A consequence of Theorem~\ref{thm: main thm} and 
\cite[Theorem 1.4.2]{GrafKovacsDuBois} is the same result for Du Bois pairs, which is stated in Corollary~\ref{cor: DuBois case}.

In the special case of a projective lc pair $(X,D)$ with globally free sheaf $\Omega_X^{[1]}(\log \lfloor D\rfloor)$, the main result of \cite{WinkelmannTrivialLogBundle} on compact K\"ahler manifolds with trivial logarithmic tangent bundle has direct implications for the geometry of $(X,D)$:

\begin{cor}[cf.\ Corollary~\ref{cor: global lc}]
Let $(X,D)$ be an lc pair such that $X$ is projective.
Then the sheaf $\Omega_X^{[1]}(\log \lfloor D\rfloor )$ is free if and 
only of there is a semi-abelian variety $T$ which acts on~$X$ with $X\setminus  \lfloor D\rfloor$ as an open orbit.
\end{cor}

\subsection*{Outline of the article}
First some definitions and notation in the context of pairs and logarithmic vector fields and $1$-forms are recalled in Section~\ref{section: notation}. In Section~\ref{section: methods}, some facts about extension of logarithmic differential forms, flows of vector fields on varieties, and residues of logarithmic $1$-forms are collected.

The case of dlt pairs is considered in Section~\ref{section: dlt}.
In Section~\ref{section: locally free with closed forms} we study pairs whose sheaf of reflexive logarithmic $1$-forms is locally generated by closed forms, and use globalisation techniques in order to obtain local embeddings into toric varieties.
Finally, the case of lc pairs is considered in Section~\ref{section: lc pairs}.
The statement for lc pairs in Theorem~\ref{thm: main thm}~(b)
is proven by reducing to part (a) of the theorem.
If the singular locus of an lc pair $(X,D)$ consists of isolated points, we prove that locally there exist closed reflexive logarithmic $1$-forms spanning the sheaf of logarithmic $1$-forms; see 
Proposition~\ref{prop: isolated lc case}.
Then an argument using hyperplane sections is used to reduce to this case and thus the setting as in Theorem~\ref{thm: main thm}~(a).

\subsection*{Acknowledgements}
I would like to thank Stefan Kebekus for introducing me to the topic and for fruitful mathematical discussions.

\newpage
\section{Definitions and Notation}\label{section: notation}

\begin{convention}
Throughout, we work over the complex numbers and all varieties are complex algebraic varieties. We will also work with the induced structure of a complex space on an algebraic variety, 
and open neighbourhoods are allowed to be open neighbourhoods in the analytic $\mathbb C$-topology.
\end{convention}

\subsection{Pairs}
In the following, we fix the notation for a few important definitions in the context of pairs. Definitions and more details may be found in \cite[Chapter~2]{KollarMori}.

\begin{defi}[Pair]\label{defi: pair}
A pair $(X,D)$ is a pair consisting of a normal quasi-projective complex variety~$X$ and a divisor $D=\sum_{i} a_iD_i$, where $D_1,\ldots , D_k$ are distinct prime divisors, $a_i\in \mathbb Q$, and $0\leq a_i\leq 1$.

The rounddown $\lfloor D\rfloor$ of the divisor $D$ is defined as
$\lfloor D\rfloor =\sum_i \lfloor a_i\rfloor D_i$.
The pair $(X,D)$ is called snc if $X$ is smooth and $D$ is snc, that is, all intersections
$D_{i_1}\cap\ldots D_{i_k}$ are smooth.

The singular locus $Z=(X,D)_{\mathrm{sing}}$ of a pair $(X,D)$ is the smallest closed subset $Z\subset X$ such $(X\setminus Z,D|_{X\setminus Z})$ is snc.
\end{defi}

Note that Definition~\ref{defi: pair} is slightly less general than \cite[Definition~2.25]{KollarMori} since we put the additional assumption that 
$0\leq a_i\leq 1$ on the coefficients $a_i$ of the divisor $D$.


\begin{notation}
We use the abbreviations klt, plt, dlt, and lc for Kawamata log terminal, purely log terminal, divisorially log terminal and log terminal. For definitions of these notions see \cite[Definition 2]{KollarMori}.
\end{notation}

\begin{defi}[Log resolution]
 Let $(X,D)$ be a pair.
 A log resolution of $(X,D)$ is a proper surjective birational morphism $\pi:\tilde X\rightarrow X$ defined on a smooth variety $\tilde X$
 such that its exceptional divisor $E=\mathrm{Exc}(\pi)$ is of pure codimension $1$ and 
 the divisor $\mathrm{Exc}(\pi)+\overline{D}$ is snc, where $\overline{D}$ is the strict transform of $D$, and $E=\mathrm{Exc}(\pi)$ is endowed with the induced reduced structure.

 We will furthermore only consider log resolutions which are strong in the sense that
 $\pi$ induces an isomorphism $\tilde X\setminus (\pi^{-1}(Z))\rightarrow X\setminus Z$
 outside the singular set $Z=(X,D)_\mathrm{sing}$ of $(X,D)$.
\end{defi}

%

%

\subsection{Logarithmic $1$-forms}
The notion of a logarithmic $1$-form is essential for this article.

\begin{notation}[Sheaves of $1$-forms]
Let $(X,D)$ be a pair.
We denote the sheaf of K\"ahler differential $1$-forms on~$X$ by $\Omega_X^1$ and the sheaf of reflexive differential $1$-forms by $\Omega_X^{[1]}$.

The sheaf of K\"ahler logarithmic $1$-forms is denoted by $\Omega_X^1(\log \lfloor D\rfloor)$ and the sheaf of reflexive logarithmic $1$-forms by $\Omega_X^{[1]}(\log\lfloor D\rfloor)$.
\end{notation}

\begin{rmk}
 We have $\Omega_X^{[1]}=(\Omega_X^1)^{**}=\iota_*(\Omega_{X_{\mathrm{reg}}}^1)$, where $\iota:X_{\mathrm{reg}}\hookrightarrow X$ denotes the inclusion of the smooth locus $X_{\mathrm{reg}}$ of~$X$.
 A reflexive $1$-form on an open subset $U\subseteq X$
is thus simply given by a $1$-form on the smooth part $U\cap X_{\mathrm{reg}}$.

Similarly, a reflexive logarithmic $1$-form on $U$ is given by logarithmic $1$-form on $U'=U\setminus (U,\lfloor D\rfloor|_U)_{\mathrm{sing}}$.
Recall that a rational $1$-form $\sigma$ on $U'$ is logarithmic 
if $\sigma$ is regular on $U'\setminus\lfloor D\rfloor$ 
and $\sigma$ and $d\sigma$ have at most first order poles along each irreducible component of $\lfloor D\rfloor$.

On $X\setminus\lfloor D\rfloor$ the notion of reflexive $1$-forms and of reflexive logarithmic $1$-forms coincide and 
we have $\Omega_X^{[1]}\cong \Omega_X^{[1]}(\log \lfloor D\rfloor)$.
\end{rmk}

\begin{convention}
Throughout the article, we always consider reflexive (logarithmic) $1$-forms and thus
a (logarithmic) $1$-form shall always mean a
reflexive (logarithmic)$1$-form.
\end{convention}

\subsection{Vector fields}
Dual to the notion of $1$-forms, there is the notion of vector fields on a variety:

\begin{notation}[Tangent sheaf]
 Recall that a vector field on a normal variety $X$ is a $\mathcal O_X$-linear derivation $\mathcal O_X\rightarrow \mathcal O_X$ of sheaves. 
 We denote the sheaf of vector fields on $X$ (or tangent sheaf of $X$) by~$\mathcal T_X$,
 and the sheaf of logarithmic vector fields (or logarithmic tangent sheaf) on a pair $(X,D)$ by $\mathcal T_X(-\log \lfloor D\rfloor)$.
\end{notation}

\begin{rmk}
A vector field on a normal variety $X$ could also be defined to be a vector field on the smooth locus $X_{\mathrm{reg}}$ of $X$, and $\mathcal T_X$ as $\mathcal T_X=\iota_*(T_{X_{\mathrm{reg}}})$ if $\iota:X_{\mathrm{reg}}\hookrightarrow X$ denotes again the 
inclusion of the smooth locus.

The sections of the logarithmic tangent sheaf $\mathcal T_X(-\log \lfloor D\rfloor)$  of a pair $(X,D)$ are precisely those vector fields
on $X$ whose flows (in the sense of Section~\ref{subsection: flows}) stabilise $\lfloor D\rfloor$ as a set.

Vector fields and $1$-forms are dual, and we have
$\mathcal T_X=(\Omega_X^1)^{*}=(\Omega_X^{[1]})^*$ and $\mathcal T_X(-\log \lfloor D\rfloor)=(\Omega_X^{[1]}(\log \lfloor D\rfloor))^*$.
In particular, the logarithmic tangent sheaf $\mathcal T_X(-\log \lfloor D\rfloor)$ is (locally) free if and only if
$\Omega_X^{[1]}(\log \lfloor D\rfloor))$ is locally free.
\end{rmk}

\section{Methods}\label{section: methods}
\subsection{Extension of differential forms}
The extension of logarithmic forms on pairs to log resolutions 
is an important tool. The following result was proven in \cite{GKKP11}:
\begin{thm}[Extension Theorem, {\cite[Theorem 1.5]{GKKP11}}]\label{thm: extension}
Let $(X,D)$ be an lc pair, $\pi:\tilde X\rightarrow X$ a log resolution, and let $\tilde D$ be the largest reduced divisor contained in the support of $\pi^{-1}(W)$, where $W$ is the smallest closed subset such that $(X\setminus W, D|_{X\setminus D})$ is klt.
Then the sheaf $\pi_*\Omega_{\tilde X}^p(\log \tilde D)$ is reflexive for any $p\leq \dim X$.\qed
\end{thm}

This means that logarithmic forms defined on the regular part of the pair $(X,D)$, i.e.\ the largest open subset $Y\subseteq X$ such that $Y$ is smooth and $D|_Y$ is an snc divisor, extend to any log resolution.

\subsection{Flows of vector fields on varieties}\label{subsection: flows}

A useful result when studying vector fields on complex varieties is the following theorem by Kaup on the existence of local flows of holomorphic vector fields on complex spaces:
\begin{thm}[Existence of flows, {\cite[Satz 3]{Kaup}}]\label{thm: existence of flows}
 Let $X$ be a normal complex space and $\xi$ a holomorphic vector field on $X$. Then the local flow of $\xi$ exists, in other words, there is 
 an open subset $\Omega\subseteq \mathbb C\times X$ such that
 \begin{enumerate}
  \item the set $\Omega $ contains $\{0\}\times X$ and for each $x\in X$, $\Omega_x=\{t\in\mathbb C\,|\, (t,x)\in \Omega\}\subseteq \mathbb C$ is connected, and
  \item there exists a holomorphic map 
  $\varphi:\Omega \rightarrow X$ 
  with
  $\varphi(0,x)=x$ for all $x\in X$ and $\frac{d}{dt}f(\varphi(t,-))=\xi(f)(\varphi(t,-))$ for any holomorphic function $f$ defined on an open subset of~$X$.\qed
 \end{enumerate}
\end{thm}

Even though vector fields on a variety $X$ can in general not be pulled back by a morphism $f:Y\rightarrow X$ to vector fields on $Y$,
the existence of local flows allows us to lift vector fields on a variety to the functorial resolution of singularities
as in \cite[Theorem 3.35, 3.36]{KollarSingularities}.
A~detailed description of this procedure can be found in \cite[$\S$ 4.2]{GKK10}.

\begin{prop}\label{prop: log resolution}
 Let $(X,D)$ be a pair and $\pi:\tilde X\rightarrow X$ the functorial log resolution of the pair.
 Let $E=\mathrm{Exc}(\pi)$ denote the exceptional divisor of $\pi$ and set $\tilde D=E+\overline {D}$, where $\overline{D}$ is the strict transform of $D$.
 Then we have
 $$\mathcal T_X\cong \pi_*(\mathcal T_{\tilde X})\cong \pi_*(\mathcal T_{\tilde X}(-\log E))\ \ \ \text{ and }\ \ \ \mathcal T_X(-\log \lfloor D\rfloor )\cong \pi_*(\mathcal T_{\tilde X}(-\log \lfloor \tilde D\rfloor)). \eqno\qed$$
 \end{prop}
 
 \begin{rmk}
   Proposition~\ref{prop: log resolution} means that every vector field $\xi$ on $X$ can be lifted to $\tilde X$ in the sense that there is a vector field $\tilde \xi$ on $\tilde X$ whose restriction to
 $\tilde X\setminus (\pi^{-1}((X,D)_{\mathrm{sing}}))\cong X\setminus (X,D)_{\mathrm{sing}}$ coincides with the restriction of $\xi$ to $X\setminus (X,D)_{\mathrm{sing}}$.
 The flow of $\tilde \xi$ stabilises the exceptional divisor $E$ and is thus logarithmic with respect to $E$.
 
 If a vector field $\xi$ on $X$ is logarithmic with respect to $D$, then $\tilde \xi$ is logarithmic with respect to $\tilde D=E+\overline {D}$.
\end{rmk}

The idea for the proof of this Proposition~\ref{prop: log resolution} is to consider the local flow of the given vector field on $X$, which exists by Theorem~\ref{thm: existence of flows}. Since the functorial resolution commutes with smooth morphisms and as the flow map $\varphi:\Omega\rightarrow X$ is a smooth morphism, it can be lifted to a local action $\tilde\varphi$ on $\tilde X$, which then induces a vector field $\tilde \xi$ on $\tilde X$. 
The flow map of a vector field on an algebraic variety is not necessarily algebraic, but in general a holomorphic map of complex spaces, one also needs to consider resolution of complex spaces at this point, see e.g.\ \cite[Theorem 3.45]{KollarSingularities}, and do the procedure for these. \\[1.5pt]
As a consequence of Proposition~\ref{prop: log resolution} and the extension result for logarithmic $1$-forms, we get the following:

\begin{cor}\label{cor: locally free}
Let $(X, D)$ be an lc pair and $\pi:\tilde X\rightarrow X$ the functorial log resolution, denote its exceptional divisor by $E$ and set $\tilde D=E+\overline{D}$, where $\overline{D}$ is the strict transform of $D$.
Let $U\subseteq X$ be an open subset such that the restriction of the sheaf $\Omega_X^{[1]}(\log  \lfloor D\rfloor)$ to $U$ is free (or equivalently, the the restriction of $\mathcal T_X(-\log  \lfloor D\rfloor)$ to $U$ is free).
Then the sheaves $\Omega_{\tilde X}^{[1]}(\log \lfloor\tilde D\rfloor)$ and $\mathcal T_{\tilde X}(-\log\lfloor \tilde D\rfloor)$ are free when restricted to $\pi^{-1}(U)$.
\end{cor}

\begin{proof}
 Since $\Omega_X^{[1]}(\log  \lfloor D\rfloor)$ and $\mathcal T_X(-\log  \lfloor D\rfloor)$ are dual to each other, one of them is (locally) free if and only if the other one is.
 Since the question is local, we may assume that $\Omega_X^{[1]}(\log  \lfloor D\rfloor)$ is free.
 
 Let $\sigma_1,\ldots,\sigma_n$ be logarithmic $1$-forms on $X$ spanning $\Omega_X^{[1]}(\log  \lfloor D\rfloor)$ and let $\xi_1,\ldots, \xi_n$ be logarithmic vector fields which span $\mathcal T_X(-\log  \lfloor D\rfloor)$
 and are dual to $\sigma_1,\ldots, \sigma_n$, i.e.\ $\sigma_i(\xi_j)=\delta_{ij}$.
 By Theorem~\ref{thm: extension}, the logarithmic $1$-forms $\sigma_1,\ldots,\sigma_n$ can be extended to logarithmic $1$-forms 
 $\tilde\sigma_1,\ldots, \tilde\sigma_n$ on $\tilde X$. Let $\tilde \xi_1,\ldots,\tilde \xi_n$ denote the lifts of $\xi_1,\ldots, \xi_n$ to $\tilde X$.
 Since $\pi $ is an isomorphism onto its image when restricted to $\pi^{-1}(X\setminus (X,D)_{\mathrm{sing}})$, we have 
 $\tilde \sigma_i(\tilde \xi_j)=\sigma_i(\xi_j)=\delta_{ij}$ on the open dense subset $\pi^{-1}(X\setminus (X,D)_{\mathrm{sing}})\cong  (X\setminus (X,D)_{\mathrm{sing}})$
of $\tilde X$ and thus on all of $\tilde X$.
 Consequently, the logarithmic $1$-forms $\tilde \sigma_1,\ldots, \tilde \sigma_n$ span $\Omega_{\tilde X}^{[1]}(\log\lfloor  \tilde D\rfloor)$, and 
$\tilde \xi_1,\ldots , \tilde \xi_n$ span $\mathcal T_{\tilde X}(-\log\lfloor \tilde  D\rfloor)$.
\end{proof}

\subsection{Residues of logarithmic 1-forms}

If $D$ is a smooth hypersurface in a complex manifold $X$, then the residue map for logarithmic $1$-forms with respect to $D$ gives an 
exact sequence 
$$0\rightarrow \Omega_X^1\rightarrow \Omega_X^1(\log D)\rightarrow 
\mathcal O_D\rightarrow 0.$$
This can directly be generalised to the case of an snc divisor $D$ in a complex manifold or smooth variety (and also logarithmic $p$-forms).
In general however, a residue sequence like this for logarithmic differential forms on an arbitrary pair does not exist.

If $(X,D)$ is a pair such that $X$ is smooth and $D$ is the sum of irreducible prime divisors, the residue of a logarithmic $p$-form can be defined as described in \cite[$\S 2$]{Saito}, but it is in general not holomorphic but meromorphic. 

\begin{prop}[{\cite[Section 1.1 and Lemma~2.2]{Saito}}]\label{prop: Saito residue}
Let $X$ be a complex manifold, $D$~a hypersurface in $X$ locally defined by 
the reduced equation $h(z)=0$ for a holomorphic function~$h$. 
If $\sigma$ is a logarithmic $1$-form on $X$, then locally there are
holomorphic functions~$g_1$,~$g_2$, and a holomorphic $1$-form $\eta$ such
that 
$$ g_1\sigma =g_2\frac{dh}{h}+\eta.$$

The functions $g_1$ and $g_2$ are in general not unique,
but the restriction to $D$ of their ratio~$\frac{g_2}{g_1}$ gives rise to a well-defined meromorphic function $\mathrm{res}(\sigma)$ on the normalisation $\tilde D$ of $D$.\qed
\end{prop}

This allows us to define a residue map as follows:
\begin{defi}
Let $X$ be a complex manifold, $D$ a reduced hypersurface in $X$, and let $\rho:\tilde D\rightarrow D$ denote the normalisation of $D$.
We define the residue map as 
$$\Omega_X^1(\log D)\rightarrow \rho_*(\mathcal M_{\tilde D}),
\ \ \ \sigma\mapsto \mathrm{res}(\sigma),$$
where $\mathcal M_{\tilde D}$ denotes the sheaf of meromorphic functions on $\tilde D$.
\end{defi}

\begin{rmk}
Similarly to the residue of a logarithmic $1$-form, the residue of logarithmic $p$-forms can be defined,
and for any logarithmic form $\sigma$ we have
$$\mathrm{res}(d\sigma)=d(\mathrm{res}(\sigma)).$$
\end{rmk}

\begin{rmk}
Recently, precise characterisations under which a reduced divisor $D$ in a complex manifold $X$ is normal crossing under the assumption that $\mathcal T_X(-\log D)$ is locally free were given in \cite{GrangerSchulze} and \cite{Faber}.
One of these equivalent characterisations is the regularity of the residue of logarithmic $1$-forms along the divisor $D$.
\end{rmk}

It turns out that also in the case of pairs $(X,D)$, where $X$ is allowed to be singular, this notion is useful.
We are always assuming that $X$ is normal and thus there is a closed subset $Z\subset X$ of codimension at least $2$ such 
$X\setminus Z$ is smooth and $D|_{X\setminus Z}$ is an snc divisor. Given any logarithmic $1$-form $\sigma$ on $X$, we may then define its residue by first restricting $\sigma$ to $X\setminus Z$, then taking the residue along $D|_{X\setminus Z}$
which then defines a unique rational function on the normalisation $\tilde D$ of $D$.

In general this residue will not be regular, nor does there exist a short exact residue sequence as in the case of snc pairs.
In the case of dlt pairs however, we have the following result for logarithmic $1$-forms:

\begin{thm}[Residue sequence for dlt pairs, {\cite[Theorem 11.7]{GKKP11}}]
Let $(X,D)$ be a dlt pair with $\lfloor D\rfloor \neq \emptyset$ and 
$D_0\subseteq \lfloor D\rfloor $ an irreducible component.
Then there is a sequence 
$$0\longrightarrow \Omega_X^{[1]}(\log(\lfloor D \rfloor - D_0))
\longrightarrow \Omega_X^{[1]}(\log  \lfloor D\rfloor  )
\overset{\mathrm{res}_{D_0}}\longrightarrow\mathcal O_{D_0}\longrightarrow 0,$$
which is exact on $X$ outside a subset of codimension at least $3$. 
Moreover this sequence coincides with the usual residue sequence where
the pair $(X,\lfloor D\rfloor)$ is an snc pair.\qed
\end{thm}

\begin{rmk}\label{rmk: dlt components}
If $(X,D)$ is a dlt pair, $p\in  D$, then by definition either the pair $(X,D)$ is snc near $p$, or there is an open neighbourhood $U\subseteq X$ of $p$ such that $(U, D|_U)$ is plt. 
Thus if $(X,D)$ is not locally snc at $p\in D$, 
we know by \cite[Proposition 5.51]{KollarMori} that $\lfloor D\rfloor $ is normal when restricted to $U$ and the disjoint union of its irreducible components. In particular, $p$ is contained in only one irreducible component of $\lfloor D\rfloor $.
\end{rmk}

In the case of lc pairs the extension of logarithmic forms to resolutions also yields residues for logarithmic $1$-forms:

\begin{rmk}\label{rmk: residues for lc pairs}
Let $(X,D)$ be an lc pair. Since logarithmic $1$-forms extend to
logarithmic $1$-forms on a log resolution of $(X,D)$ by \cite[Theorem 1.5]{GKKP11}, the residue of a logarithmic $1$-form along a component of $\lfloor D\rfloor $ is a regular function on the normalisation of that component of  $\lfloor D\rfloor $.
Moreover, we have an exact sequence
$$0\rightarrow \Omega_X^{[1]}\rightarrow 
\Omega_X^{[1]}(\log\lfloor D\rfloor )\rightarrow 
\bigoplus_{j=1}^{k} (\rho_j)_*(\mathcal O_{\tilde D_j}),$$
where $D_1,\ldots, D_k$ denote the irreducible components of 
the rounddown $\lfloor D\rfloor$ and $\rho_j:\tilde D_j\rightarrow D_j$ is the normalisation of $D_j$.
Note however that the last arrow of this sequence is in general not surjective.
\end{rmk}

\begin{rmk}\label{rmk: residues of closed forms}
Let $(X,D)$ be an arbitrary pair, and $\sigma$  a logarithmic $1$-form on $X$ that  is closed. Since $d(\mathrm{res}_{D_j}(\sigma))=\mathrm{res}_{D_j}(d\sigma)=0$ along each irreducible component $D_j$ of $\lfloor D\rfloor$, the residue $\mathrm{res}(\sigma)$ is constant on each irreducible component~$D_j$.
\end{rmk}

\section{Dlt pairs with locally free sheaf of logarithmic $1$-forms}\label{section: dlt}
If $(X,D)$ is a dlt pair and its sheaf of logarithmic differential $1$-forms is locally free, then $(X,D)$ is necessarily snc:

\begin{thm}\label{thm: dlt pairs}
Let $(X,D)$ be a dlt pair such that 
 $\Omega_X^{[1]}(\log  \lfloor D\rfloor)$ is locally free.
Then $X$ is smooth and $ \lfloor D\rfloor $ is an snc divisor.
\end{thm}

\begin{proof}
After shrinking~$X$, we may assume that $\Omega_X^ {[1]}(\log 
 \lfloor D\rfloor)$ is free.
If $p\notin  \lfloor D\rfloor $, then the pair $(X,D)$ is klt near~$p$ and 
$\Omega_X^ {[1]}(\log 
 \lfloor D\rfloor )\cong \Omega_X^{[1]}$ near~$p$.
Since the Lipman-Zariski conjecture is true for klt spaces by  \cite[Theorem~6.1]{GKKP11}, $X$~is smooth near~$p$.

Assume now $p\in  \lfloor D\rfloor $
and let $\pi:\tilde X\rightarrow X$ be the functorial log resolution with exceptional divisor~$E$, $\overline{D}$ the strict transform of $D$ and $\tilde D=E+\overline{D}$.  Then 
$\Omega_{\tilde X}^ {[1]}(\log\lfloor  \tilde D\rfloor )$ is free
by Corollary~\ref{cor: locally free}.
Suppose that the pair $(X,D)$ is not snc at $p$. Then by Remark~\ref{rmk: dlt components} the point $p$ is only contained in one irreducible component of $ \lfloor D\rfloor $, and after possibly further shrinking
we may thus assume that $ \lfloor D\rfloor$ is irreducible.
Let $\overline{ D }$ denote again the strict transform of $ D$, and $E_1,\ldots , E_m$ the exceptional divisors,
$\tilde D =\overline{D }+E_1+\ldots +
E_m$.
Let $\sigma_1,\ldots, \sigma_n$ be logarithmic $1$-forms spanning $\Omega_X^ {[1]}(\log 
\lfloor D\rfloor)$, and let $\tilde\sigma_1,\ldots , \tilde \sigma_n$
denote the extensions of these to $\tilde X$ (cf.\ Theorem~\ref{thm: extension}), which span $\Omega_{\tilde X}^ {[1]}(\log \lfloor \tilde D\rfloor)$.
Let $q_0\in \pi^{-1}(p)\cap \lfloor \overline{D}\rfloor$.
Since $\tilde D$ is snc, there is ~$j$  such that the residue $\mathrm{res}_{\lfloor\overline{D}\rfloor}(\tilde \sigma_j)$ along $\lfloor\overline{D}\rfloor$ does not vanish in $q_0$
and hence $\mathrm{res}_{\lfloor D\rfloor }(\sigma_j)(p)\neq 0$.
Each component of $\lfloor D\rfloor$ is normal since $(X,D)$ is dlt and we may thus assume that $\mathrm{res}_{\lfloor{D}\rfloor}(\sigma_1)=1$ and $\mathrm{res}_{\lfloor{D}\rfloor}(\sigma_i)=0$ for $i>1$.
Therefore, $\sigma_2,\ldots,\sigma_n$ are regular $1$-forms without poles. By  \cite[Theorem~3.1]{GrafKovacs} the extensions $\tilde\sigma_2,\ldots, \tilde\sigma_n$ of $\sigma_2,\ldots ,\sigma_n$ to $\tilde X$ are also regular, and have in particular no poles along the exceptional divisors $E_1,\ldots,E_m$, hence $\mathrm{res}_{E_i}(\tilde\sigma_j)=0$ for $j>1$ and all $i$.
Choose $l\in\{1,\ldots, k\}$ such that $E_l$ intersects $\lfloor\overline{D}\rfloor$, which exists since we supposed that
$(X,D )$ is not snc at $p\in \lfloor D \rfloor$.
Let $q\in E_l\cap \lfloor \overline{D}\rfloor$ and $\eta $ be a logarithmic $1$-form on a neighbourhood of~$q$ such $\mathrm{res}_{E_l}(\eta)=1$ and 
$\mathrm{res}_{\lfloor\overline{ D}\rfloor}(\eta)=0$.
By Corollary~\ref{cor: locally free} we have $\eta=\alpha_1\tilde\sigma_1
+\ldots +\alpha_n\tilde \sigma_n$ for some regular functions $\alpha_j$.
This implies $$0=\mathrm{res}_{\lfloor \overline{D}\rfloor}(\eta)
=\alpha_1|_{\lfloor \overline{D}\rfloor }\mathrm{res}_{\lfloor \overline{D}\rfloor}(\tilde \sigma_1)+\ldots +\alpha_n|_{\lfloor \overline{D}\rfloor}\mathrm{res}_{\lfloor \overline{D}\rfloor}(\tilde\sigma_n)=\alpha_1|_{\lfloor \overline{D}\rfloor},$$
and in particular $\alpha_1(q)=0$.
But then we also have 
$$\mathrm{res}_{E_l}(\eta)=\alpha_1|_{E_l}\mathrm{res}_{E_l}(\tilde \sigma_1)+\ldots +\alpha_n|_{E_l}\mathrm{res}_{E_l}(\tilde\sigma_n)
=\alpha_1|_{E_l}$$
and in particular $\mathrm{res}_{E_l}(\eta)(q)=\alpha_1(q)
\mathrm{res}_{E_l}(\tilde \sigma_1)(q)=0$, which is a contradiction to
${\mathrm{res}_{E_l}(\eta)=1}$.
\end{proof}

\section{Pairs with locally free sheaf of logarithmic $1$-forms  generated by closed forms.}\label{section: locally free with closed forms}
If we allow slightly more general singularities for the pair $(X,D)$ than dlt singularities, e.g.\ if $(X,D)$ is lc, then the statement of Theorem~\ref{thm: dlt pairs} is no longer true.
Even if the sheaf of logarithmic $1$-forms is locally free, $X$ could have singularities or the irreducible components of 
$\lfloor D\rfloor $ could be non-normal:

\begin{ex}\label{ex: nodal curve}
Let $X=\mathbb A^2$ and $D=\{y^2-x^3-x^2=0\}$ be the nodal curve, which is not normal. The pair $(X,D)$ is lc and its sheaf of logarithmic $1$-forms is locally free.
\end{ex}

\begin{ex}\label{ex: cusp}
 Let $X=\mathbb A^2$ and $D=\{y^2-x^3=0\}$ be the cusp. In this case the pair $(X,D)$ is not lc, $D$ is not normal, but the sheaf of logarithmic $1$-forms is locally free.
\end{ex}

\begin{ex}\label{ex: toric variety}
Let $X$ be a normal toric variety. Let $T\subseteq X$ denote the open orbit of the $(\mathbb C^*)^n$-action and 
set $D=X\setminus T$. Then the pair $(X,D)$ is lc by 
\cite[Proposition~3.7]{KollarSingOfPairs}.

Moreover, the sheaves of logarithmic vector fields on $(X,D)$ and logarithmic $1$-forms can be described rather explicitly (see e.g.\ 
\cite[$\S$ 3.1]{OdaToricGeometry}), and in particular
these are free sheaves.
\end{ex}

In this section, we consider the case of a pair $(X,D)$ whose 
sheaf $\Omega_X^{[1]}(\log \lfloor D\rfloor )$ of logarithmic $1$-forms is locally free and locally generated by closed $1$-forms.
 
The main result (see  Theorem~\ref{thm: closed one forms}) is that pairs consisting of a toric variety and boundary divisor as in Example~\ref{ex: toric variety} describe the local structure of all such pairs,
i.e.\ if $(X,D)$ is a pair whose sheaf $\Omega_X^{[1]}(\log \lfloor D\rfloor )$ of logarithmic $1$-forms is locally free and locally generated by closed $1$-forms, then $(X,\lfloor D\rfloor)$ is toroidal.

\begin{rmk}\label{rmk: X smooth outside D}
Let $X$ be a normal complex space.
Then any closed $1$-form $\sigma $ on the smooth locus of $X$ extends to any resolution of singularities of $X$ by \cite[Theorem 1.2]{JoerderWLZ}.
As a consequence the Lipman-Zariski conjecture holds for normal complex spaces $X$ whose sheaf $\Omega_X^{[1]}$ is locally free and locally generated by closed $1$-forms, see \cite[Theorem 1.1]{JoerderWLZ}.

For the case of a pair $(X,D)$ whose 
sheaf $\Omega_X^{[1]}(\log \lfloor D\rfloor )$ of logarithmic $1$-forms is locally free and locally generated by closed $1$-forms
this directly implies that $X\setminus \lfloor D\rfloor $ is smooth.
\end{rmk}

Next, we show that the requirement to locally have a basis for 
$\Omega_X^{[1]}(\log \lfloor D\rfloor )$ consisting of closed forms and the requirement that $\Omega_X^{[1]}(\log \lfloor D\rfloor )$ is locally free and locally generated by closed forms are equivalent.

\begin{lemma}
Let $(X,D)$ be a pair such that $\Omega_X^{[1]}(\log \lfloor D\rfloor )$ is locally free and locally generated by closed $1$-forms.
Then locally there exists a basis of closed $1$-forms $\sigma_1,\ldots ,\sigma_n$ spanning $\Omega_X^{[1]}(\log \lfloor D\rfloor )$.
\end{lemma}

\begin{proof}
After possibly shrinking $X$, let $\sigma_1,\ldots,\sigma_n$ be a basis of logarithmic $1$-forms, and let $\tau_1,\ldots,\tau_m$ be closed $1$-forms which generate $\Omega_X^{[1]}(\log \lfloor D\rfloor )$.
Since $\sigma_1,\ldots, \sigma_n$ form a basis, there is 
an $m\times n$-matrix $A$ whose entries $a_{ij}$ are regular functions on $X$ and such that 
$$\left(\begin{smallmatrix}\tau_1\\ \vdots\\ \tau_m\end{smallmatrix}\right)=A\left(\begin{smallmatrix}\sigma_1\\ \vdots\\ \sigma_n\end{smallmatrix}\right).$$
Similarly, since $\tau_1,\ldots,\tau_m$ generate $\Omega_X^{[1]}(\log \lfloor D\rfloor )$, there is an $n\times m$-matrix $B$ whose entries are regular functions and such that
$$\left(\begin{smallmatrix}\sigma_1\\ \vdots\\ \sigma_n\end{smallmatrix}\right)=B\left(\begin{smallmatrix}\tau_1\\ \vdots\\ \tau_m\end{smallmatrix}\right).$$
Combining the above, we get
$$\left(\begin{smallmatrix}\sigma_1\\ \vdots\\ \sigma_n\end{smallmatrix}\right)=B\left(\begin{smallmatrix}\tau_1\\ \vdots\\ \tau_m\end{smallmatrix}\right)
=BA\left(\begin{smallmatrix}\sigma_1\\ \vdots\\ \sigma_n\end{smallmatrix}\right)$$
and because $\sigma_1,\ldots, \sigma_n$ form a basis we get
$BA=\mathrm{id}$. In particular, the matrix $B$ has rank $n$ at each point, and (after possibly reordering) $\tau_1,\ldots,\tau_n$ form a local basis for $\Omega_X^{[1]}(\log \lfloor D\rfloor )$.
\end{proof}

\begin{lemma}\label{lemma: closed equals commuting}
Let $(X,D)$ be a pair such that $\Omega_X^{[1]}(\log \lfloor D\rfloor )$ is locally free. Let $\sigma_1,\ldots,\sigma_n$ be a local basis of the logarithmic $1$-forms and let $\xi_1,\ldots,\xi_n$ be a dual local basis  of logarithmic vector fields for $\mathcal T_X(-\log \lfloor D\rfloor )$. 
Then the $1$-forms $\sigma_1,\ldots \sigma_n$ are closed if and only if $\xi_1,\ldots,\xi_n$ pairwise commute, i.e.\ $[\xi_i,\xi_j]=0$ for all $i,j$.
\end{lemma}

\begin{proof}
On the smooth locus of any variety we have 
$$d\sigma(\xi,\xi')
=\xi(\sigma(\xi'))-\xi'(\sigma(\xi))-\sigma([\xi,\xi'])$$ for any regular $1$-form $\sigma$ and arbitrary vector fields $\xi,\xi'$.
Therefore, we get
\begin{align*}
d\sigma_i(\xi_j,\xi_k)&=\xi_j(\sigma_i(\xi_k))
-\xi_k(\sigma_i(\xi_j))-\sigma_i([\xi_j,\xi_k])\\
&=\xi_j(\delta_{ik})-\xi_k(\delta_{ij})-\sigma_i([\xi_j,\xi_k])\\
&=-\sigma_i([\xi_j,\xi_k])
\end{align*}
on the smooth locus of $X\setminus \lfloor D\rfloor $, and by continuity this holds on all of $X$.
Since $\sigma_1,\ldots, \sigma_n$ and $\xi_1,\ldots ,\xi_n$
are local bases for $\Omega_X^{[1]}(\log \lfloor D\rfloor )$ and $\mathcal T_X(-\log \lfloor D\rfloor )$, 
we have $d\sigma_i=0$ for any $i$ if every commutator 
$[\xi_j,\xi_k]$ vanishes and vice versa.
\end{proof}

Let us now consider the case of a $(X,D)$ with locally free sheaf $\Omega_X^{[1]}(\log \lfloor D\rfloor )$ which is locally generated by closed forms.
We start with the case where $\lfloor D\rfloor $ is irreducible.

\begin{lemma}\label{lemma: irreducible D}
Let $(X,D)$ be a pair such that $\Omega_X^{[1]}(\log \lfloor D\rfloor )$ is locally free, $\lfloor D\rfloor $ is irreducible and assume that $\Omega_X^{[1]}(\log \lfloor D\rfloor )$ is locally generated by closed $1$-forms.
Then $X$ is smooth and $\lfloor D\rfloor $ is smooth.
\end{lemma}

\begin{proof}
By Remark~\ref{rmk: X smooth outside D} we already know that 
$X\setminus \lfloor D\rfloor $ is smooth.
Let $p\in \lfloor D\rfloor \subset X$ be a singular point of $X$ and shrink $X$ such that $\Omega_X^{[1]}(\log \lfloor D\rfloor )$ is free and generated by the closed forms $\sigma_1,\ldots,\sigma_n$.
The residue of the closed forms $\sigma_j$ along $\lfloor D\rfloor $ is constant 
(cf.\ Remark~\ref{rmk: residues of closed forms}) and thus the residue of each logarithmic $1$-form along $\lfloor D\rfloor $ is regular.
Let $q\in  \lfloor D\rfloor $ be a smooth point of $X$. Then locally near $q$, the $\lfloor D\rfloor $ is given by an equation $h=0$ for a regular function~$h$. Moreover, $\sigma=\frac{dh}{h}$ defines a logarithmic $1$-form near~$q$ and $\mathrm{res}_{\lfloor D\rfloor }(\sigma)=1$.
Thus, there is $j\in\{1,\ldots, n \}$ such that 
$\mathrm{res}_{\lfloor D\rfloor }(\sigma_j)\neq 0$, and hence we may assume without loss of generality that 
$\mathrm{res}_{\lfloor D\rfloor}(\sigma_1)=1$ and $\mathrm{res}_{\lfloor D\rfloor }(\sigma_j)=0$ for $j>1$. Then $\sigma_2,\ldots,\sigma_2$ are regular.

Let $\pi:\tilde X\rightarrow X$ be the functorial log resolution of the pair $(X,D)$ as in Proposition~\ref{prop: log resolution},
let~$E$ be the exceptional divisor and $\overline{D}$ the strict transform of $D$, $\tilde D=\overline{D}+E$.
Since the $1$-forms $\sigma_2,\ldots,\sigma_n$ are regular and closed, they extend to regular $1$-forms $\tilde{\sigma}_2,\ldots,\tilde{\sigma}_n$ on $\tilde X$ by \cite[Theorem~ 1.2]{JoerderWLZ}.

Furthermore, let $\xi_1,\ldots,\xi_n$ be logarithmic vector fields which are dual to $\sigma_1,\ldots,\sigma_n$. They lift to vector fields $\tilde{\xi}_1,\ldots,\tilde{\xi}_n$ on $\tilde X$ whose flows stabilise each component of $\lfloor\tilde D\rfloor=\lfloor\overline{D}\rfloor+E$ as explained in Proposition~\ref{prop: log resolution}.
We may thus restrict these vector fields to $\lfloor\overline{D}\rfloor$. Since 
$$(\tilde{\sigma}_i|_{\lfloor\overline{D}\rfloor})(\tilde{\xi_j}|_{\lfloor\overline{D}\rfloor})=\tilde\sigma_i(\tilde\xi_j)=\sigma_i(\xi_j)=\delta_{ij}$$ for any $i,j\geq 2$, the vector fields
$\tilde{\xi}_2|_{\lfloor\overline{D}\rfloor},\ldots,\tilde{\xi}_n|_{\lfloor\overline{D}\rfloor}$ are independent at each point in $\lfloor\overline{D}\rfloor$. Their flows also stabilise $E$ and thus $E\cap \lfloor\overline{D}\rfloor$ which yields a contradiction as $n-2=\dim E\cap\lfloor\overline{D}\rfloor$.
\end{proof}

If $(X,D)$ is any pair such that $\Omega_X^{[1]}(\log \lfloor D\rfloor)$ is locally free, then it does not follow in general that the irreducible components of $\lfloor D\rfloor$ are normal as illustrated in Example~\ref{ex: nodal curve}.
However, if we also assume that $\Omega_X^{[1]}(\log \lfloor D\rfloor)$ is locally generated by closed forms such examples cannot occur.

\begin{prop}\label{prop: normality}
Let $(X,D)$ be a pair such that $\Omega_X^{[1]}(\log \lfloor D\rfloor)$ is locally free and locally generated by closed forms. 
Let $D_1,\ldots ,D_k$ be the irreducible components of $\lfloor D\rfloor$.
Then for any subset $I\subseteq \{1,\ldots, k\}$ the intersection $$\bigcap_{i\in I} D_i$$
is normal.
\end{prop}

The proposition is a consequence of the following lemma, which describes the local geometry of group actions induced by appropriate logarithmic vector fields.

\begin{lemma}\label{lemma: geometry of globalisation}
Let $(X,D)$ be a pair such that $\Omega_X^{[1]}(\log \lfloor D\rfloor)$ is locally free and locally generated by closed forms. Let $D_j$ be an irreducible component of $\lfloor\overline{D}\rfloor$ and $p\in D_j$.
Then there is a neighbourhood $U$ of $p$ such that the following is true:
\begin{enumerate}
 \item There is a basis of closed logarithmic $1$-forms $\sigma_1,\ldots,\sigma_n$ for $\Omega_X^{[1]}(\log \lfloor D\rfloor)$ on $U$ such that 
 $\mathrm{res}_{D_j}(\sigma_1)=2\pi i$ and $\mathrm{res}_{D_j}(\sigma_k)=0$ for all $k\neq 1$.
 \item Let $\xi_1,\ldots, \xi_n$ be a basis of logarithmic vector fields on $U$ dual to $\sigma_1,\ldots,\sigma_n$.
 Then there is an $S^1$-action $\varphi: S^1\times U\rightarrow U$ on $U$ which induces the vector field $\xi_1$, i.e.\ 
$\left.\frac{d}{dt}\right|_{t=1}{(f\circ\varphi)(t,x)}=\xi_1(f)(x)$ for any $x\in U$ and any holomorphic function $f$ defined near $x$.
\item There is an open embedding $\iota:U\hookrightarrow Y\subseteq \mathbb C^N$ into a normal Stein space $Y$ such that 
there is a holomorphic $\mathbb C^*$-action $\psi:\mathbb C^*\times Y\rightarrow Y$ which is induced by a linear $\mathbb C^*$-action on $\mathbb C^N$ and
induces the $S^1$-action $\varphi$ on $U$, i.e.\ $\psi|_{S^1\times U}=\varphi$, where we identify $U$ and $\iota(U)$.
\item Let $A=\{y\in Y\,|\, \psi(t,y)=y \text{ for all } t\in\mathbb C^ *\}$ be the fixed point set of the $\mathbb C^ *$-action on~$Y$, and 
let $\pi:Y\rightarrow Y/\!/\mathbb C^*$ be the categorical quotient of $Y$ by the $\mathbb C^*$-action $\psi$.
Then $A=U\cap D_j$ and the quotient space $Y/\!/\mathbb C^*$ is isomorphic to $A$.
\item Let $B\subseteq U$ be a closed analytic subset which is $S^1$-invariant, i.e.\ $S^1\cdot B = \varphi(S^1\times B) =B$.
Then $\mathbb C^*\cdot B = \psi(\mathbb C^*\times B)$ is a closed subset of $Y$ and $(\mathbb C^*\cdot B )\cap U= B$. Moreover,
if $B$ is normal, then $B\cap A$ is normal.
\end{enumerate}
\end{lemma}

Before proving the lemma, we show how the above proposition follows from the lemma.
\begin{proof}[Proof of Proposition \ref{prop: normality}]
Relabelling the components of $\lfloor D\rfloor$ if necessary, it is enough to show that 
if $D_1\cap\ldots \cap D_{j-1}$ is normal, then $D_1\cap \ldots \cap D_j$ is normal.

Let $p\in D_1\cap\ldots \cap D_j$. 
Let $U\subseteq X$ be an open neighbourhood of $p$ as described in the preceding lemma, $Y\subseteq \mathbb C^N$ a normal complex Stein space with a $\mathbb C^*$-action
$\psi:\mathbb C^*\times Y\rightarrow Y$, and $\iota:U\rightarrow Y$ an embedding such that
the restriction of the vector field $\xi$ induced by the $\mathbb C^*$-action $\psi$ to $U$ is a logarithmic vector field with respect to $D$ and such that
there is a local basis $\sigma_1,\ldots ,\sigma_n$ for $\Omega^{[1]}_X(\log \lfloor D\rfloor)|_U$ consisting of closed forms such that $\xi,\xi_2,\ldots,\xi_n$ is a dual basis for 
$\mathcal T_X(-\log \lfloor D\rfloor)|_U$ for appropriate logarithmic vector fields $\xi_2,\ldots,\xi_n$ on $U$.

Again, we identify $U$ with its image $\iota(U)\subseteq Y$ and let $\pi:Y\rightarrow Y/\!/\mathbb C^*$ denote the categorical quotient. 
As before the quotient $Y/\!/\mathbb C^*$ may be identified with set the of fixed points $A=D_j\cap U$ of the $\mathbb C^*$-action.

Set $B=D_1\cap \ldots \cap D_{j-1}\cap U$. Then $B$ is a closed analytic subset of $U$ which is $S^1$-invariant by construction. Moreover, the set $B$ is normal by assumption.
Then by part (5) of the preceding lemma the intersection $B\cap A=B\cap D_j=D_1\cap \ldots\cap D_{j-1}\cap D_j$ is normal.
\end{proof}

\begin{proof}[Proof of Lemma \ref{lemma: geometry of globalisation}]

By Lemma~\ref{lemma: irreducible D} we already know that
$D_j\setminus\left(\bigcup_{i\neq j} D_i\right)$ is smooth for each~$j$.
Since the question is local, we may assume that 
$\Omega_X^{[1]}(\log \lfloor D\rfloor)$ is free and spanned by closed forms
$\sigma_1,\ldots, \sigma_n$.
By the same argument as used in the proof of Lemma~\ref{lemma: irreducible D}, we may furthermore assume that $\mathrm{res}_{D_j}(\sigma_1)=2\pi i$ and
$\mathrm{res}_{D_j}(\sigma_i)=0$ for $i>1$. This proves part (1).
\vspace{5.5pt}

In order to prove part (2), let $\xi_1,\ldots,\xi_n$ be a basis of logarithmic vector fields dual to $\sigma_1,\ldots,\sigma_n$.
Let $\chi:\Omega\rightarrow X$ be the flow map of $\xi_1$ (cf.\ Theorem~\ref{thm: existence of flows}), $\Omega\subseteq \mathbb C\times X$. 
Let $q\in D_j\setminus  \left(\bigcup_{i\neq j} D_j\right)$. Then $X$ and~$D_j$  are smooth near $q$ by Lemma~\ref{lemma: irreducible D}.
We may now define local coordinates on a suitable neighbourhood of $q$ by setting
$$z_1(x)=\exp\left(\int_{q_0}^x \sigma_1\right)$$
and 
$$z_i(x)=\int_{q_0}^x \sigma_i$$
for $i>1$ and a fixed point $q_0\in X\setminus \lfloor D\rfloor$ near $q$.
Note that the integrals are independent of the chosen path since $\sigma_1,\ldots,\sigma_n$ are closed, $\mathrm{res}_{D_j}(\sigma_1)=2\pi i$ and $\sigma_2,\ldots,\sigma_n$ are holomorphic near~$q$.
With respect to these coordinates we have $D_j=\{z_1=0\}$ and
$$\sigma_1=d\log (z_1)
=\frac{dz_1}{z_1}, \sigma_2=dz_2,\ldots, \sigma_n=dz_n,$$
and the dual vector fields $\xi_1,\ldots,\xi_n$ are thus 
necessarily of the form 
$$\xi_1=z_1\frac{\partial}{\partial z_1},
\xi_2=\frac{\partial}{\partial z_2},\ldots, \xi_n=\frac{\partial}{\partial z_n}.$$
Therefore $\xi_1$ vanishes along $\{z_1=0\}$ and by the identity principle along all of $D_j$.
Since each point in $D_j$ is a fixed point of the flow $\chi:\Omega\rightarrow X$ of $\xi_1$, 
there is an open connected neighbourhood $V$ of $p$ such that the domain $\Omega$ of definition of $\chi$ can be chosen such that
$$((-1,1)\times (-4\pi,4\pi))\times V\subset \Omega \subseteq \mathbb C\times X,$$
and such that $V\cap D_j$ is connected and contains $p$ and $q$.

The flow of $\xi_1$ with respect to the local coordinates $z_1,\ldots,z_n$ is given by 
$$\chi\left(t,\left(\begin{smallmatrix}z_1\\ \vdots \\ z_n
\end{smallmatrix}\right)\right)=\left(\begin{smallmatrix}e^t z_1\\ z_2\\ \vdots \\ z_n
\end{smallmatrix}\right).$$
Thus $\chi(2\pi i,x)=x$ for all $x$ near $q$ and by the identity principle we get $\chi(2\pi i,x)=x$ for all $x\in V$.
Let $V'$ be a relatively compact open subset of $V$ such that
$V'\cap D_j$ is also connected and contains $p$ and such 
that $\chi (\{0\}\times(-4\pi,4\pi)\times V')\subseteq V$.
Consequently, $\chi(i(s+t),x)$ and $\chi(is,\chi(it,x))$ are defined for all $x\in V'$, $s,t\in (-4\pi,4\pi)$ with $s+t\in (-4\pi,4\pi)$, and
we have $$\chi(i(s+t),x)=\chi(is,\chi(it,x)).$$
Set $U=\chi (\{0\}\times(-4\pi,4\pi)\times V')$. Then $U$ is an open neighbourhood of $p$, $U\subseteq V$, and we can define 
a map $\varphi:S^1\times U\rightarrow U$ by setting
$$\varphi(e^{it},x)=\chi(it,x).$$
This is well-defined since $\chi(2\pi i,x)=x$ for all $x\in V$, $\chi(i(s+t),x)=\chi(is,\chi(it,x))$ implies $\varphi(S^1\times U)\subseteq U$ and that $\varphi$ is a group action. Moreover, this $S^1$-action $\varphi$ induces $\xi_1$ by construction,
and thus we proved $(2)$.
\vspace{5.5pt}

By standard arguments (see for example \cite[2.3~Proposition]{Fischer}), there is an open neighbourhood $U'\subseteq U$ of $p$ and an embedding $\iota :U'\rightarrow \mathbb C^N$, 
and moreover we can choose $N$ minimal in the sense that 
$N=\dim T_p(U')=\dim T_p(X)$. 
We may assume $\iota(p)=0$. 
Consider now the set 
$$U''=\bigcap_{s\in S^1}\varphi(\{s\}\times U')\subseteq U'.$$
This set $U''$ is open and contains $p$ since $S^1$ is compact and $p$ a fixed point because $p\in D_j$ and $\xi_1|_{D_j}=0$. Moreover, $U''$ is $S^1$-invariant, 
i.e.\ $\varphi(S^1\times U'')=U''$.
After shrinking, we may thus assume that $U=U'=U''$.
Moreover, we will always identify $U$ and $\iota (U)\subseteq \mathbb C^N$ in the following and also denote the inclusion map $U=\iota(U)\hookrightarrow \mathbb C^N$ by $\iota$. 

Since $p$ is a fixed point of the $S^1$-action $\varphi$, we get a linear $S^1$-action on $T_pX\cong \mathbb C^N$ by differentiation, which we denote by $\rho:S^1\times\mathbb C^N\rightarrow \mathbb C^N$.
Next, we want to average the embedding $\iota:U\hookrightarrow\mathbb C^N$ in order to obtain an embedding $U\hookrightarrow \mathbb C^N$ which is equivariant with respect to the $S^1$-action $\varphi$ on $U$ and the linear $S^1$-action $\rho$ on $\mathbb C^N$.
Let $\mu$ denote the normalised Haar measure on $S^1$ and
set 
$$\tilde\iota:U\rightarrow \mathbb C^N,\ \ 
\tilde \iota(u)=\int_{S^1} \rho(s)\iota (\varphi(s^{-1},u))d\mu(s).$$
Then $\tilde\iota(p)=0$ and $\tilde\iota$ is equivariant by construction.
Identifying  $T_pU\cong \mathbb C^N$ and $T_0(\iota(U))=T_0\mathbb C^N\cong \mathbb C^N$ appropriately, we have $D\iota(p)=\mathrm{id}_{\mathbb C^N}$ and thus
$$D\tilde\iota(p)=
\int_{S^1} D\rho(s) D\iota(p)D\varphi(s^{-1},p)d\mu(s)
=\int_{S^1} \rho(s)\circ \mathrm{id}\circ\rho(s^{-1})d\mu(s)
=\mathrm{id}$$
and consequently $\tilde \iota$ is an immersion at $p$.
Thus we can shrink $U$ and get an equivariant embedding $\tilde\iota:U\hookrightarrow \mathbb C^N$, 
and identifying $U$ and $\tilde\iota (U)\subseteq\mathbb C^N$, the $S^1$-action $\varphi$ on $U$ is induced by a linear $S^1$-action on $\mathbb C^N$.
After possibly further shrinking~$U$ and rescaling, we may assume that $U$ is a closed subset of the open unit ball~$B_N=\{z\in\mathbb C^N\,|\,\langle z, z\rangle <1\}$
with respect to an $S^1$-invariant hermitian inner product $\langle \, , \rangle$.

The linear $S^1$-action on $\mathbb C^N$ extends to a linear 
$\mathbb C^*$-action $\psi:\mathbb C^*\times\mathbb C^N\rightarrow \mathbb C^N$ on $\mathbb C^N$, 
and the restriction of the induced vector field $\left.\frac{d}{dt}\right|_{t=1}\psi(t,-)$ to $U$ is precisely the vector field $\xi_1$.

The function $\alpha:\mathbb C^N\rightarrow \mathbb R$, 
$z\mapsto \langle z,z\rangle $, is $S^1$-invariant and plurisubharmonic. Therefore, the open unit ball
$B_N=\{z\in\mathbb C^N\,|\, \alpha(z)<1\}$ is orbit-convex (cf.\ \cite[\S 3.4 Proposition]{HeinznerInvariantTheory}),  i.e.\ 
for every $z\in B_N$ and $v\in \mathbb R=i\mathrm{Lie}(S^1)$ such 
that $\exp(v)\cdot z=\psi(\exp(v),z)\in B_N$ we also
have $\exp(tv)\cdot z=\psi(\exp(tv),z)\in B_N$ for all 
$t\in[0,1]$.

Define now $Y=\mathbb C^*\cdot U=\psi(\mathbb C^*\times U)\subseteq \mathbb C^N$.
Then $Y$ is an irreducible normal complex space since $U$ is normal and $U\subseteq Y$ is an open subset.
By \cite[{\S}~3.3]{HeinznerInvariantTheory}, the complex space $Y$ is the $S^1$-complexification (in the sense of \cite[{\S}~1.1]{HeinznerInvariantTheory}) of the $S^1$-invariant analytic subset $U$ of the open unit ball $B_N$ and
consequently, $Y$ is a Stein space by \cite[{\S}~6.6]{HeinznerInvariantTheory}.
This finishes the proof of part (3).
\vspace{5.5pt}

The categorical quotient $\pi:Y\rightarrow Y/\!/\mathbb C^*$ of $Y$ by the $\mathbb C^*$-action $\psi$ exists and is a complex Stein space by \cite[Theorem~5.3]{Snow}.
Furthermore, $Y/\!/\mathbb C^* $ is normal since $Y$ is normal (see \cite{Snow}, Lemma~3.2 and the subsequent remark).

By definition of $Y$ as $Y=\mathbb C^*\cdot U=\psi(\mathbb C^ *\times U)$, the fixed point set $$A=\{y\in Y\,|\,\psi(t,y)=y\text{ for all }t\in\mathbb C^*\}$$ is contained in $U$.
For elements $u\in U$ we know that $u\in A$ precisely if $\xi_1(u)=0$, and thus we get $D_j\cap U\subseteq A$.

Let $q\in D_j\cap U$, $q\notin \bigcup_{i\neq j} D_i$, such that $q$ is a smooth point of $U$ and~$D_j$.
As argued before, there are local coordinates $z_1,\ldots, z_n$ for $U$ near $q$ such that locally $D_j=\{z_1=0\}$ and
$\xi_1=z_1\frac{\partial}{\partial z_1}$, and locally near $q$ the set of fixed points $A$ and $D_j$ coincide.

Moreover, $q$ is an attractive fixed point of the $\mathbb C^*$-action, i.e.\ there is a neighbourhood $W\subseteq Y$ of $q$ such that for any $y\in W$ the closure of the orbit $\mathbb C^*\cdot y$ through $y$ contains a fixed point.
Then the set of fixed points $A$ is a closed irreducible subspace of~$Y$ by \cite[Theorem 6.2]{Snow} and hence $A=D_j\cap U$.

Since every fibre of $\pi$ contains precisely one closed orbit, and the set of fixed points $A$ is the set of orbits of minimal dimension, $\pi(A)$ is closed. Moreover, $\pi(A)$ is open since there is an attractive fixed point. Therefore, we get that $A$ is isomorphic to $Y/\!/\mathbb C^*$ 
and every fixed point is an attractive fixed point, see also \cite[Theorem~ 6.2]{Snow}.
\vspace{5.5pt}

In order to prove part (5) of the lemma,
let $B\subseteq U$ be a closed analytic $S^1$-invariant subset of~$U$.
The results of \cite[{\S}~3.3]{HeinznerInvariantTheory} now directly imply that $\mathbb C^*\cdot B=\psi (\mathbb C^*\times B)$ is a closed analytic subset of $Y$, and 
$(\mathbb C^*\cdot B)\cap U=B$.

In particular, $\mathbb C^*\cdot B$ is a complex Stein space, and 
normal if $B$ is normal.
Let $A$ denote again the set of fixed point of the $\mathbb C^*$-action on $Y$. 
By similar arguments as used before, it follows that
the categorical quotient $(\mathbb C^*\cdot B)/\!/\mathbb C^*$, which is normal if $\mathbb C^ *\cdot B$ is normal, can be identified with the set of fixed points $A'$ in $\mathbb C^*\cdot B$,
and we have $$A'=A'\cap U=A\cap (\mathbb C^*\cdot B)\cap U=A\cap B.$$
This shows in particular that $A\cap B$ is normal if $B$ is normal
\end{proof}

\begin{lemma}\label{lemma: extension for dual forms}
 Let $(X,D)$ be a pair. Let $\sigma_1,\ldots,\sigma_k$ be closed $1$-forms on $X\setminus \lfloor D\rfloor$ 
and $\sigma_{k+1},\ldots, \sigma_n$ closed $1$-forms on $X$ such that the sheaf $\Omega_{X}^{[1]}|_{X\setminus \lfloor D\rfloor}$ is spanned by $\sigma_1,\ldots,\sigma_n$.
 Let $\xi_1,\ldots,\xi_n$ be vector fields on $X$ which are logarithmic with respect to $\lfloor D\rfloor$ and dual to $\sigma_1,\ldots, \sigma_n$ (on $X\setminus \lfloor D\rfloor$), i.e.\ $\sigma_i(\xi_j)=\delta_{ij}$.
 Assume that the vector fields $\xi_1,\ldots, \xi_k$ are induced by $S^1$-actions, i.e.\ there are actions $\psi_j:S^1\times X\rightarrow X$ of the Lie group~$S^1$ by holomorphic transformations such 
 that the induced vector field $\left.\frac{d}{d s}\right|_{s=1} \psi_j(s,\cdot )$ coincides with~$\xi_j$.
 
 Then the $1$-forms $\sigma_1,\ldots ,\sigma_k$ extend to logarithmic $1$-forms $\sigma_1,\ldots,\sigma_k\in\Omega_X^{[1]}(\log \lfloor D\rfloor)(X)$.
\end{lemma}

\begin{proof}
 Since the pair $(X,\lfloor D\rfloor)$ is snc outside a set of codimension at least $2$ and $\Omega_X^{[1]}(\log \lfloor D\rfloor)$ is reflexive, it is enough to consider the case where $X$ is smooth and $\lfloor D\rfloor$ an snc divisor.
By Lemma~\ref{lemma: geometry of globalisation} the vector fields $\xi_1,\ldots,\xi_n$ pairwise commute.
 Hence, the $S^1$-actions $\psi_j$ all commute and thus induce an $(S^1)^k$-action 
 $\psi:(S^1)^n\times X\rightarrow X$ by setting $$\psi\left(\left(\begin{smallmatrix}  s_1\\ \vdots\\  \\ s_k \end{smallmatrix} \right), p\right) =\psi_1(s_1,\psi_2(s_2,\ldots \psi_n(s_k,p)\ldots ) )$$
 for $\left(\begin{smallmatrix}   s_1\\ \vdots\\  \\ s_k \end{smallmatrix}   \right)\in (S^1)^k$, $p\in X$.
 
 Let $p_0\in \lfloor D\rfloor$.
First, we consider the case where $k=n$ and $p_0$ is a fixed point of the $(S^1)^n$-action~$\psi$, or equivalently $\xi_1(p_0)=\ldots =\xi_n(p_0)=0$.
 Then the action can locally be linearised, i.e.\ there are local coordinates $z_1,\ldots, z_n$ near $p_0$ such 
 that $p_0=0$ and the action $\psi$ is linear in these coordinates. 
 Moreover, we may assume that $z_1,\ldots, z_n$ are chosen such that there are constants $a_{ij}$ for $i,j=1,\ldots, n$
 such that 
 $$\psi\left(\left(\begin{smallmatrix}  s_1\\ \vdots\\  \\ s_n  \end{smallmatrix} \right), \left(\begin{smallmatrix}  z_1\\ \vdots \\ \\ z_n  \end{smallmatrix} \right)\right) =
 \left(\begin{smallmatrix}
        s_1^{a_{11}}\cdot\ldots \cdot s_n^{a_{n1}}& & \\
        & \ddots & \\
        &&\\
        & & s_1^{a_{1n}}\cdot\ldots\cdot s_n^{a_{nn}}
       \end{smallmatrix}\right)\cdot 
\left(\begin{smallmatrix}  z_1\\ \vdots \\ \\ z_n  \end{smallmatrix} \right)
=\left(\begin{smallmatrix}  s_1^{a_{11}}\cdot\ldots \cdot s_n^{a_{n1}}\cdot z_1\\ \vdots\\  \\ s_1^{a_{1n}}\cdot\ldots\cdot s_n^{a_{nn}} z_n  \end{smallmatrix} \right)
   $$
 and then
 $$\xi_i(z)=\sum_{j=1}^n a_{ij}z_j\frac{\partial}{\partial z_j}.$$

 Since $\sigma_1,\ldots, \sigma_n$ and hence $\xi_1,\ldots,\xi_n$ are linearly independent on $X\setminus \lfloor D\rfloor$ we get 
 that the matrix $A=(a_{ij})_{1\leq i,j\leq n}$ has to be invertible.
 We may thus replace the vector fields $\xi_1,\ldots, \xi_n$ (and also $\sigma_1,\ldots, \sigma_n$) by an invertible linear combination of them and 
 then get $\xi_j=z_j\frac{\partial}{\partial z_j}$ for all $j=1,\ldots n$.
 This implies $\sigma_j=\frac{1}{z_j}dz_j$, and therefore $\sigma_1,\ldots, \sigma_n$ are logarithmic $1$-forms.
 
 Let now $p_0\in \lfloor D\rfloor$ be any point and $k$ be arbitrary. 
 Let 
 $G=((S^1)^k)_{p_0}=\{s\in (S^1)^k\, |\, \psi(s,p_0)=p_0\}$ denote the isotropy group in $p_0$.
 Since $G$ is a closed subgroup of $(S^1)^k$, we have $G\cong (S^1)^l$ for some $l\leq k$, 
 and since $(S^1)^k$ is abelian there is a Lie subgroup $H\cong (S^1)^{k-l}$ of $(S^1)^k$ 
 such that $(S^1)^k\cong G\times H$.
 After possibly again replacing the vector fields $\xi_1,\ldots,\xi_k$ by an invertible linear combination of them we may assume 
 that the Lie algebra of $G$ is spanned by $\xi_1,\ldots, \xi_{l}$ and the Lie algebra of $H$ by $\xi_{l+1},\ldots, \xi_k$.
Since the isotropy group of $H$ in $p_0$ is trivial by construction and the vector fields $\xi_j$ all satisfy
$\sigma_i(\sigma_j)=\delta_{ij}$ on $X$ for $i>k$, we get that $\xi_{l+1},\ldots,\xi_n$ are independent at each point in a neighbourhood of~$p_0$
 
  Moreover, $\xi_1,\ldots,\xi_n$ are commuting and thus span an involutive distribution (of rank~$n-l$)
locally near $p_0$. Therefore, by Frobenius' theorem there are local coordinates $w_1,\ldots,w_n$ such that 
$$\xi_{l+1}=\frac{\partial}{\partial w_{l+1}},\ldots, \xi_n=\frac{\partial}{\partial w_n}$$ and we may assume also 
$w_{1}(p_0)=\ldots =w_n(p_0)=0$.
 
 Now, we want to average $w_{l+1},\ldots,w_n$ in order to obtain $G$-invariant coordinates.
 We define
 $$z_j(p)=\int_{G} w_j(\psi(s,p))d \mu(s)$$
 for $j>l$ and where $\mu$ denotes the normalised Haar measure on $G\cong (S^1)^l$.
 Setting $z_1=w_1,\ldots, z_l=w_l$ we get new coordinates 
 $z_1,\ldots, z_n$ such 
 that $z_1(p_0)=\ldots=z_n(p_0)=0$,
 $\xi_{l+1}=\frac{\partial}{\partial z_{l+1}},\ldots, \frac{\partial}{\partial z_n}$ and such that
 $z_{l+1},\ldots, z_n$ are $G$-invariant.
 Consequently, the subset $S=\{z_{l+1}=\ldots,z_n=0\}$ is $G$-invariant. We have $p_0\in S$ by construction and may now apply the above argument to $S$ and the $G\cong (S^1)^ l$-action on $S$.
\end{proof}

\begin{cor}\label{cor: extension of closed log forms}
 Let $(X,D)$ be a pair such that $\Omega_X^{[1]}(\log \lfloor D\rfloor)$ is locally free and locally generated by closed forms.
 Let $\pi:\tilde X\rightarrow X$ be any log resolution of the pair $(X,D)$, let $E$ be the exceptional divisor of $\pi$ and $\overline{D}$ the strict transform of $D$.
 
 Then every logarithmic $1$-from on $(X,D)$ extends to a logarithmic $1$-from on $(\tilde X,\tilde D)$, where $\tilde D=E+\overline{D}$.
\end{cor}

\begin{proof}
Let $D_1,\ldots, D_k$ denote the irreducible components of $\lfloor D\rfloor$, $\lfloor D\rfloor=D_1+\ldots +D_k$. 
Since the statement is local, it is enough to prove the statement for a neighbourhood of a point $p\in D_1\cap\ldots\cap D_k$.
Let  $\sigma_1,\ldots, \sigma_n$ be closed logarithmic $1$-forms which span $\Omega_X^{[1]}(\log \lfloor D\rfloor)$ in a neighbourhood of $p$.

Let $$A=(\mathrm{res}_{D_i}(\sigma_j))_{1\leq i\leq k, 1\leq j\leq n}$$ be
the $k\times n$-matrix whose entry at the position $(i,j)$ is
the residue of $\sigma_j$ along the divisor~$D_i$.
Note that by Remark~\ref{rmk: residues of closed forms} all entries of $A$ are complex numbers.
After relabelling the indices of the $D_i$'s and passing to a linear combination of $\sigma_1,\ldots,\sigma_n$
we may assume that there is $l\leq n$ such that 
$\mathrm{res}_{D_i}(\sigma_i)=2\pi i$ and $\mathrm{res}_{D_i}(\sigma_j)=0$ for all $i\leq l$ and all $j\neq i$, and
 $\mathrm{res}_{D_i}(\sigma_j)=0$ for all $j>l$, i.e.\ 
$\sigma_{l+1},\ldots,\sigma_n$ are regular.
By \cite[Theorem~1.2]{JoerderWLZ} we already know then that
$\sigma_{l+1},\ldots, \sigma_n$ extend to regular $1$-forms on $\tilde X$.

Let $\xi_1,\ldots,\xi_n$ be logarithmic vector fields dual to $\sigma_1,\ldots,\sigma_n$.
For any $i\leq l$ we have $\mathrm{res}_{D_i}(\sigma_i)=2\pi i$ and 
$\mathrm{res}_{D_i}(\sigma_j)=0$ for $j\neq i$,
and thus by Lemma~\ref{lemma: geometry of globalisation}, (2) we get 
that for $i\leq l$ there are open neighbourhoods $U_i$ of $p$ and $S^1$-actions 
$\varphi_i:S^1\times U_i\rightarrow U_i$ which induce $\xi_i$.

Let $\tilde{\xi}_1,\ldots,\tilde{\xi}_n$ denote the lifts of $\xi_1,\ldots,\xi_n$ to $\tilde X$ (cf.\ Proposition~\ref{prop: log resolution}).
The $S^1$-actions $\varphi_i$ also lift to $\tilde X$ and induce the vector fields $\tilde{\xi}_1,\ldots,\tilde{\xi}_l$.
Moreover, the vector fields $\tilde{\xi}_{l+1},\ldots, \tilde{\xi}_n$ are independent at each point since $\sigma_{l+1},\ldots,\sigma_n$ extend to regular $1$-forms on $\tilde X$.
An application of Lemma~\ref{lemma: extension for dual forms} now yields that $\sigma_1,\ldots,\sigma_l$ extend to logarithmic $1$-forms on $\tilde X$.
\end{proof}

\begin{lemma}\label{lemma: intersection is smooth}
Let $(X,D)$ be a pair such that $\Omega_X^{[1]}(\log \lfloor D\rfloor)$ is free 
and generated by the closed logarithmic $1$-forms $\sigma_1,\ldots,\sigma_n$ such that $\sigma_{l+1},\ldots,\sigma_n$ are regular.
Let $D_1,\ldots ,D_k$ be the irreducible components of $\lfloor D\rfloor$.

Then $\sigma_{l+1},\ldots,\sigma_n$ can be restricted to any intersection $D_{i_1}\cap\ldots\cap D_{i_j}$ for $i_1,\ldots, i_j\in\{1,\ldots,k\}$, i.e.\ there are regular $1$-forms $\eta_{l+1},\ldots,\eta_n$ on $D_{i_1}\cap \ldots \cap D_{i_j}$ such that $\iota^*(\sigma_i)=\eta_i$ if $\iota:D_{i_1}\cap \ldots\cap D_{i_j}\hookrightarrow X$ denotes the inclusion map.

Moreover, we have $\dim(D_1\cap\ldots\cap D_k)\geq n-l$ 
if $D_1\cap\ldots\cap D_k\neq\emptyset$,
and $D_1\cap\ldots\cap D_k$ is smooth.
\end{lemma}

\begin{proof}
Without loss of generality, let $D_{i_1}\cap \ldots\cap D_{i_j}=D_1\cap \ldots\cap D_j$, and suppose $D_1\cap\ldots\cap D_j\neq \emptyset$.
Recall that by Proposition~\ref{prop: normality} this intersection $D_1\cap\ldots\cap D_j$ is normal.

Let $\xi_1,\ldots,\xi_n$ be a basis of logarithmic vector fields dual to $\sigma_1,\ldots,\sigma_n$.
Since the flows of $\xi_1,\ldots,\xi_n$ stabilise each irreducible component of $\lfloor D\rfloor$, the vector fields $\xi_1,\ldots,\xi_n$ induce vector fields on $D_1\cap\ldots\cap D_j$ for any $j\leq k$.

Let $\pi:\tilde X\rightarrow X$ be a log resolution of $(X,D)$ with exceptional divisor $E$ and let 
$\overline{D}_i$ be the strict transform of $D_i$.
By Corollary~\ref{cor: extension of closed log forms}, $\sigma_1,\ldots,\sigma_n$ extend to logarithmic $1$-forms $\tilde{\sigma}_1,\ldots,\tilde{\sigma}_n$ on $\tilde X$.

We first want to restrict to $D_1$. We may assume 
$\mathrm{res}_{D_1}(\sigma_1)=1$ and $\mathrm{res}_{D_1}(\sigma_i)=0$ for $i>1$.
 Then also $\mathrm{res}_{\overline{D}_1}(\tilde{\sigma}_i)=0$ for $i>1$ and 
$\tilde{\sigma}_i$ is thus regular along $\overline{D}_1\setminus 
(E\cup \overline{D}_2\cup\ldots\cup \overline{D}_k)$.
Therefore the restriction of $\tilde{\sigma}_2,\ldots,\tilde{\sigma}_n$ to $\overline{D}_1$ yields logarithmic $1$-forms with respect to $(E+ \overline{D}_2+\ldots+ \overline{D}_k)|_{\overline{D}_1}$.
Since $D_1$ is normal, $D_1$ is isomorphic to an open subset of $\overline{D}_1$ outside a closed subset of codimension at least $2$.
Consequently, the logarithmic $1$-forms  $\tilde{\sigma}_2|_{\overline{D}_1},\ldots,\tilde{\sigma}_n|_{\overline{D}_1}$ on
the strict transform $\overline{D}_1$ induce logarithmic $1$-forms on $(D_1,(D_2+\ldots+D_k)|_{{D}_1})$
since $\Omega_{D_1}^{[1]}(\log (D_2+\ldots+D_k)|_{D_1})$ is reflexive, and these give the desired restrictions 
$\sigma_2|_{D_1},\ldots,\sigma_n|_{D_1}$ of $\sigma_2,\ldots,\sigma_n$ to $D_1$.
Moreover, the restricted vector fields $\xi_2|_{D_1},\ldots,\xi_n|_{D_1}$ are dual to these and 
hence $\sigma_2|_{D_1},\ldots,\sigma_n|_{D_1}$ yield a basis of logarithmic $1$-forms on $D_1$.
The vector fields $\xi_2|_{D_1},\ldots,\xi_n|_{D_1}$ commute and 
$\sigma_2|_{D_1},\ldots,\sigma_n|_{D_1}$ are closed.

If $k>1$ and if there is $D_i$, say $D_2=D_i$, with $\mathrm{codim}_{D_1}(D_1\cap D_i)=1$, and we apply the procedure again.
It might now happen that $D_i\cap (D_1\cap D_2)=D_1\cap D_2$ for some $i>2$. Assume $(D_1\cap D_2)\cap \ldots\cap D_{i}=D_1\cap D_2$
and $(D_1\cap D_2 )\cap D_{i'}\subsetneq D_1\cap D_2$ for all $i'>i$.
In this case, $\sigma_3,\ldots,\sigma_n$ restrict to 
logarithmic $1$-forms of the pair $(D_1\cap \ldots\cap D_i,
(D_{i+1}+\ldots+D_k)|_{D_1\cap \ldots\cap D_i})$, and $\xi_3,\ldots,\xi_n$ induce dual logarithmic vector fields on $D_1\cap \ldots\cap D_i$.

We continue then iteratively. 
At each step either the boundary $(D_{i+1}+\ldots+D_k)|_{D_1\cap \ldots\cap D_{i'}}$ of the pair $(D',D_0)=(D_1\cap \ldots\cap D_{i},
(D_{i+1}+\ldots+D_k)|_{D_1\cap \ldots\cap D_{i}})$
is empty or
otherwise there is $i'>i$ such that $D_{i'}\cap (D_1\cap \ldots\cap D_{i})$ has codimension $1$ in $D_1\cap \ldots\cap D_{i}$
as explained in the following:

By construction we have $D_{i'}\cap (D_1\cap \ldots\cap D_{i})\neq D'=D_1\cap \ldots\cap D_{i}$ and thus the codimension is at least $1$. 
On $D'=D_1\cap \ldots\cap D_{i}$ we have the restricted logarithmic $1$-forms $\sigma_r|_{D'},\ldots,\sigma_n|_{D'}$ and
dual logarithmic vector fields $\xi_r|_{D'},\ldots,\xi_n|_{D'}$,
where $r\leq i+1$, $r-1=\mathrm{codim}_X (D')$.
If the codimension of $D_{i'}\cap D'$ in $D'$ is at least $2$ for all $i'>i$, then
the logarithmic $1$-forms $\sigma_r|_{D'},\ldots,\sigma_n|_{D'}$ are regular since $D'$ is normal. But this is in contradiction to the fact that the vector fields $\xi_r|_{D'},\ldots,\xi_n|_{D'}$ stabilise each $D_{i'}$.

This procedure eventually gives rise to
regular $1$-forms $\sigma_r|_{D_1\cap \ldots\cap D_k},\ldots,
\sigma_n|_{D_1\cap \ldots\cap D_k}$ on $D_1\cap \ldots\cap D_k$, $r\leq l+1$ and dual vector fields $\xi_r|_{D_1\cap \ldots\cap D_k},\ldots \xi_n|_{D_1\cap \ldots\cap D_k}$.
Hence, we have $\dim (D_1\cap \ldots\cap D_k)=n-r+1\geq n-(l+1)-1=n-l$.
Furthermore, $D_1\cap \ldots\cap D_k$ is smooth by  \cite[Theorem~1.1]{JoerderWLZ} since
 $\sigma_r|_{D_1\cap \ldots\cap D_k},\ldots ,
\sigma_n|_{D_1\cap \ldots\cap D_k}$ are a basis for $\Omega_{D_1\cap\ldots\cap D_k}^ {[1]}$ and each  $\sigma_j|_{D_1\cap \ldots\cap D_k}$ is closed.
\end{proof}

\begin{rmk}\label{rmk: residue matrix}
In the setting of the previous lemma and its proof, we also get 
that $r=\mathrm{rk}(A)+1$ and hence
$$\dim (D_1\cap \ldots\cap D_k)=n-r+1=n-\mathrm{rk}(A),$$
where 
$$A=(\mathrm{res}_{D_i}(\sigma_j))_{1\leq i\leq k,1\leq j\leq n}$$ is the matrix of residues as before.
\end{rmk}

\begin{thm}\label{thm: closed one forms}
 Let $(X,D)$ be a pair such that $\Omega_X^{[1]}(\log \lfloor D\rfloor)$ is locally free and locally generated by closed logarithmic $1$-forms.
 Then $(X,\lfloor D\rfloor)$ is toroidal, i.e.\ for any point there is a neighbourhood $U\subseteq X$ which is isomorphic to an open subset of a toric variety $Y$ with open $(\mathbb C^*)^n$-orbit $T$,
 and the divisor $\lfloor D\rfloor$ corresponds to the complement $Y\setminus T$ of $T$ in $Y$.
\end{thm}

\begin{proof}
 Let $D_1,\ldots, D_k$ denote the irreducible components of $\lfloor D\rfloor$, and let $p\in X$. 
 Since the statement of the theorem is local, we may assume that $\Omega_X^{[1]}(\log \lfloor D\rfloor)$ is free and generated by the closed logarithmic $1$-forms $\sigma_1,\ldots,\sigma_n$, 
 and that $p\in D_1\cap \ldots \cap D_k$.
 
 We first consider the case where $D_1\cap\ldots\cap D_k=\{p\}$.
 Let 
 $$A=(\mathrm{res}_{D_i}(\sigma_j))_{1\leq i\leq k, 1\leq j \leq n}$$
 denote again the residue matrix of the forms $\sigma_1,\ldots,\sigma_n$.
 Since we assumed $\dim (D_1\cap \ldots\cap D_k)=0$, we have $\mathrm{rk}(A)=n$ by Remark~\ref{rmk: residue matrix}.
 In particular, there are at least $n=\dim X$ irreducible components of $\lfloor D\rfloor$ containing the point $p$,
 and without loss of generality we may assume 
 that $A$ is of the form
 
 $$A=(\mathrm{res}_{D_i}(\sigma_j))_{1\leq i\leq k, 1\leq j \leq n}=\left(\begin{smallmatrix}
                                                                           2\pi i & & \\
                                                                           &\ddots & \\
                                                                           && 2\pi i\\ 
                                                                           &&\\
                                                                           \hline \\                                                                          
                                                                           &&\\
                                                                           &\displaystyle B & \\
                                                                           &&
                                                                          \end{smallmatrix}\right),$$
where $B$ is an arbitrary $(k-n)\times n$-matrix.
Let $\xi_1,\ldots,\xi_n$ be a basis of logarithmic vector fields dual to $\sigma_1,\ldots,\sigma_n$.
By Lemma~\ref{lemma: geometry of globalisation} (2) there is an open neighbourhood~$U_j$ of $p$ for any $j=1,\ldots, n$ 
and an
$S^1$-action $\varphi_j:S^1\times U_j\rightarrow U_j$ which induces the vector field $\xi_j$ on $U_j$ and such that $p$ is a fixed point of this $S^1$-action.

There is a neighbourhood $U'$ of $p$ such that $\varphi_1,\ldots, \varphi_n$
define a map $\varphi:(S^1)^n\times U'\rightarrow X$ by setting
$$\varphi\left(\left(\begin{smallmatrix}  s_1\\ \vdots\\  \\ s_n  \end{smallmatrix} \right), q\right) =\varphi_1(s_1,\varphi_2(s_2,\ldots \varphi_n(s_n,q)\ldots )$$
 for $\left(\begin{smallmatrix}   s_1\\ \vdots\\  \\ s_n \end{smallmatrix}   \right)\in (S^1)^n$ and $q\in U'$.
Moreover, since the vector fields $\xi_1,\ldots,\xi_n$ and hence the $S^1$-actions $\varphi_1,\ldots,\varphi_n$ all commute,
there is an open neighbourhood $U$ of $p$ such that $\varphi:(S^1)^n\times U\rightarrow U$ is a $(S^1)^n$-action; set e.g.\ 
$$U=\bigcap_{t\in (S^1)^n} \varphi(\{t\}\times U')$$
and note that $U$ is open since $(S^1)^n$ is compact and $U$ contains $p$ since $\varphi((S^1)^n\times \{p\})=\{p\}$.

By the same arguments as used in the proof of Lemma~\ref{lemma: geometry of globalisation} (3), we can shrink $U$ such that 
there is a normal Stein space $Y\subseteq \mathbb C^N$ with a holomorphic $(\mathbb C^*)^n$-action 
$\psi:(\mathbb C^*)^n\times Y\rightarrow Y$ which is induced by a linear $(\mathbb C^*)^n$-action on $\mathbb C^N$ and such that
there is an open equivariant embedding $\iota:U\hookrightarrow Y$ with $Y=\psi((\mathbb C^*)^n\times \iota(U))$, and we identify again $U$ and $\iota (U)$.
Moreover, we may assume that $U$ is a closed analytic subset of the open unit ball $B_N=\{z\in\mathbb C^N\,|\,\langle z, z\rangle <1\}$
with respect to an $(S^1)^n$-invariant hermitian inner product $\langle \, , \rangle$.

Let $q\in U\setminus \lfloor D\rfloor$ and consider the orbit $(\mathbb C^*)^n\cdot q=\psi((\mathbb C^*)^n\times \{q\})$, which is open and dense in $Y$.
The unique closed orbit in its closure 
$\overline{(\mathbb C^*)^n\cdot q}$ in the ambient space 
$\mathbb C^N$ is $0$. 
Thus every orbit $(\mathbb C^*)^n\cdot x$ with $x\in \overline{(\mathbb C^*)^n\cdot q}$ contains $0$ in its closure
and there is $x'\in B_N$ with $(\mathbb C^*)^n\cdot x=(\mathbb C^*)^n\cdot x'$.
Since $U\subset B_N$ is analytic and $(S^1)^n$-invariant
and $B_N$ is orbit-convex, we have
$((\mathbb C^*)^n\cdot U)\cap B_N=U$ (cf.\ \cite[$\S$ 3.3 Corollary]{HeinznerInvariantTheory})
and then $$((\mathbb C^*)^n\cdot q) \cap B_N\subset ((\mathbb C^*)^n\cdot U)\cap B_N=U.$$
This implies $x'\in U$ and hence $\overline{(\mathbb C^*)^n\cdot q}=Y$.
Consequently, $Y$ is an affine toric variety.

Now, let $\dim (D_1\cap \ldots\cap D_k)$ be arbitrary. The intersection $D_1\cap \ldots\cap D_k$ is smooth by Lemma~\ref{lemma: intersection is smooth} and we 
have $\dim (D_1\cap \ldots\cap D_k)=n-\mathrm{rk}(A)$ by Remark~\ref{rmk: residue matrix}.
Thus we may assume that $\sigma_{l+1},\ldots,\sigma_n$, where $l=\mathrm{rk}(A)$, are regular $1$-forms and 
$\mathrm{res}_{D_i}(\sigma_j)=2\pi i \delta_{ij}$ for $i,j\leq l$.
Let $p\in D_1\cap \ldots \cap D_k$ and let $\xi_1,\ldots,\xi_n$ denote again the logarithmic vector fields dual to $\sigma_1,\ldots,\sigma_n$. 
Applying Lemma~\ref{lemma: geometry of globalisation} (2) to the vector fields $\xi_1,\ldots,\xi_l$, 
we get commuting 
$S^1$-actions $\varphi_j:S^1\times U_j\rightarrow U_j$ on some neighbourhood $U_j$ of $p$, $j=1,\ldots, l$, which induce the vector fields $\xi_j$.
This gives now rise to 
an $(S^1)^l$-action $\varphi:(S^1)^l\times U\rightarrow U$ on some neighbourhood $U$ of $p$.
As before (cf.\ Lemma~\ref{lemma: geometry of globalisation} (3)) we may globalise the corresponding local $(\mathbb C^*)^l$-action 
and get that there are a normal complex Stein space $Y\subseteq \mathbb C^N$ with a linear $(\mathbb C^*)^l$-action $\psi:(\mathbb C^*)^l\times Y\rightarrow Y$
and an equivariant open embedding $\iota:U\hookrightarrow Y$.
Identifying $U$ and its image $\iota(U)$ we have 
that the set $A$ of fixed points of the $(\mathbb C^*)^l$-action~$\psi$ in $Y$ is precisely $A=U\cap (D_1\cap\ldots\cap D_k)$ 
and moreover $A$ is isomorphic to $Y/\!/ (\mathbb C^*)^l$ if
$\pi:Y\rightarrow Y/\!/(\mathbb C^*)^l$ denotes the categorical quotient of $Y$ by the action $\psi$.
The vector fields $\xi_{l+1},\ldots,\xi_n$ induce commuting and independent vector fields on $D_1\cap\ldots\cap D_k$, and since they also commute 
with $\xi_1,\ldots,\xi_l$, they induce 
vector fields $\hat\xi_{l+1},\ldots,\hat\xi_n$ on the quotient $Y/\!/(\mathbb C^*)^l$ with
$\xi_j\circ \pi^*=\pi^*\circ\hat\xi_j$ for $j=l+1,\ldots, n$.
The fibre $\pi^{-1}(p)$ of $p\in U\cap D_1\cap\ldots\cap D_k\cong Y/\!/(\mathbb C^*)^l$ 
is $l$-dimensional and the flows of $\xi_1,\ldots,\xi_l$ stabilise $\pi^{-1}(p)$ by construction such that 
$\xi_1,\ldots,\xi_l$ induce commuting vector fields on $\pi^{-1}(p)$.
The flows of $\xi_{l+1},\ldots,\xi_n$ and $\hat\xi_{l+1},\ldots,\hat\xi_n$ now induce a local isomorphism
$\chi:S\times X'\rightarrow X$ onto its image, where $S$ is an open neighbourhood of $p$ in $U\cap D_1\cap\ldots\cap D_k\cong Y/\!/(\mathbb C^*)^l$ 
and $X'$ an open neighbourhood of $p$ in $\pi^{-1}(p)$.
Since $U\cap D_1\cap\ldots\cap D_k$ is smooth, $S$ is also smooth and the divisors $D_1,\ldots,D_k$ induce divisors
$D_1|_{X'},\ldots,D_k|_{X'}$ on $X'$.
Moreover, we may restrict $\sigma_1,\ldots,\sigma_l$ to $X'$, they are dual to $\xi_1|_{X'},\ldots,\xi_l|_{X'}$ and thus give
rise to a basis of closed logarithmic $1$-forms $\sigma_1|_{X'},\ldots,\sigma_l|_{X'}$ of $\Omega_{X'}(\log \lfloor D\rfloor|_{X'})$.
The intersection of the divisors $D_j|_{X'}$ is now $D_1|_{X'}\cap\ldots\cap D_k|_{X'}=\{p\}$ and applying the above arguments to $X'$
we conclude that $X'$ is toroidal. Consequently, $S\times X'$, which is isormorphic to a neighbourhood of $p$ in $X$, is toroidal.
\end{proof}

\section{Lc pairs with (locally) free sheaf of logarithmic $1$-forms}\label{section: lc pairs}
In this section the case of an lc pair $(X,D)$ with (locally) free sheaf of logarithmic $1$-forms is considered.

As already noted in Examples~\ref{ex: nodal curve} and \ref{ex: toric variety} we cannot expect that $X$ is smooth in this case.
However, the singularities in these examples are contained in the support of $\lfloor D\rfloor$, and this is true in general.
Since the Lipman-Zariski conjecture holds for lc pairs (see \cite[Corollary~1.3]{GrafKovacs} or \cite[Theorem~1.1]{Druel}), we have the following:
\begin{rmk}
If $(X,D)$ is lc and $\Omega_X^{[1]}(\log \lfloor D\rfloor)$ is locally free,
then $X\setminus \lfloor D\rfloor$ is smooth since the sheaf of $1$-forms and the sheaf of logarithmic $1$-forms agree on $X\setminus  \lfloor D\rfloor$.
\end{rmk}

In the following, we first consider the case of an lc pair $(X,D)$ where $X$ is projective and the sheaf of logarithmic $1$-forms is free.
Then we deal with the case of a (not necessarily projective) lc pair $(X,D)$ with locally free sheaf of logarithmic $1$-forms. The goal is to prove that $(X,D)$ is toroidal by reducing to the case where the sheaf of logarithmic $1$-forms is spanned by closed forms 
as in the previous section.

\subsection{Lc pairs with free sheaf of logarithmic $1$-forms}

Let us consider the case of an lc pair $(X,D)$ such that its sheaf of logarithmic $1$-forms is free and assume additionally that $X$ is projective.

In the case of a smooth compact K\"ahler (or weakly K\"ahler) manifold $X$ and an snc divisor~$D$, 
Winkelmann described precisely under which conditions the logarithmic tangent bundle is trivial. In particular, the following result for smooth  projective varieties is obtained.

\begin{thm}[{\cite[Corollary 1]{WinkelmannTrivialLogBundle}}]\label{thm: thmWinkelmann}
Let $X$ be a smooth projective variety and $D$ a reduced snc divisor on $X$.
Then $\mathcal T_X(-\log D)$ is a free sheaf if and only if
there is a semi-abelian variety~$T$ acting on $X$ with $X\setminus D$ as an open orbit.\qed
\end{thm}

Recall that a semi-abelian variety is an algebraic group which is a quotient of $(\mathbb C^*)^n$ by a lattice~$\Gamma$ which contains a $\mathbb C$-basis of $\mathbb C^n$.

As a consequence of this result, we get an
explicit description of projective lc pairs $(X,D)$ with free sheaf of logarithmic $1$-forms.

\begin{cor}\label{cor: global lc}
Let $(X,D)$ be an lc pair such that $X$ is projective.
Then the logarithmic tangent sheaf $T_X(-\log \lfloor D\rfloor )$ is free if and 
only of there is a semi-abelian variety~$T$ which acts on $X$ with $X\setminus \lfloor D\rfloor $ as an open orbit.
\end{cor}

\begin{proof}
 Let $\pi:\tilde X\rightarrow X$ be a resolution of the pair $(X,D)$ as in Proposition~\ref{prop: log resolution} and denote $\tilde D=E+\overline{D}$, where $E$ is the exceptional divisor and $\overline{D}$ the strict transform of $D$.
Then by Proposition~\ref{prop: log resolution} and Corollary~\ref{cor: locally free}, the sheaf $\mathcal T_X(-\log \lfloor D\rfloor)$ is free if and only if $\mathcal T_{\tilde X}(-\log\lfloor \tilde D\rfloor)$ is free. 

Consequently, if $\mathcal T_X(-\log \lfloor D\rfloor )$ is free, then
Theorem~\ref{thm: thmWinkelmann} implies that 
there is a semi-abelian variety $T$ acting on $\tilde X$ with $\tilde X\setminus \lfloor \tilde D\rfloor$ as an open orbit.
Each component of $\lfloor\tilde D\rfloor $ and thus in particular the exceptional divisors are $T$-invariant.
Therefore, the $T$-action on $\tilde X$ induces a $T$-action on  $X$ with $X\setminus  \lfloor D\rfloor$ an open orbit. 

Conversely, if there is an action of a semi-abelian variety $T$ on $X$ with $X\setminus \lfloor D\rfloor $ as an open orbit, then
this action lifts to $\tilde X$ by 
\cite[Proposition~3.9.1]{KollarSingularities} with 
$\tilde X\setminus \lfloor\tilde D\rfloor $ as an open orbit.
Consequently,  $\mathcal T_{\tilde X}(-\log\lfloor \tilde D\rfloor)$ and hence $\mathcal T_X(-\log \lfloor D\rfloor)$ are free. 
\end{proof}

\subsection{Lc pairs with locally free sheaf of logarithmic $1$-forms}
We now consider the case of an arbitrary lc pair $(X,D)$ whose logarithmic tangent sheaf 
$\mathcal T_X(-\log \lfloor D\rfloor)$ is locally free.

First, we deal with the isolated case in the sense that
there is a point at which every logarithmic vector field vanishes.
Then the general case is considered and reduced to isolated case by an inductive argument via hyperplane sections.

\begin{lemma}\label{lemma: exceptional divisor is toric}
 Let $(X,D)$ be an lc pair with locally free tangent sheaf
 $\mathcal T_X(-\log \lfloor D\rfloor)$.
Suppose that there is $p\in X$ such 
that $\xi(p)=0$ for all logarithmic vector fields $\xi$ defined on some neighbourhood of $p$.

Then there exists a log resolution $\pi:\tilde X\rightarrow X$  of the pair $(X, D)$ with exceptional divisor~$E$ with the following properties:
 \begin{enumerate}
  \item[(1)] Each irreducible component of $\pi^{-1}(p)$ is a toric variety.
  \item[(2)] There is a point $q\in\pi^{-1}(p)$ such that $\xi(q)=0$ for any logarithmic vector field 
  $\xi\in T_{\tilde X}(-\log\lfloor \tilde D\rfloor)(U)$ defined on some open neighbourhood $U\subseteq \tilde X$ of $q$,
 where $\tilde D=\overline{D}+E$ for the strict transform $\overline{D}$ of $D$.
 \end{enumerate} 
\end{lemma}

\begin{proof}[Proof of (1)]
Shrink $X$ such $\mathcal T_X(-\log\lfloor D\rfloor )$ is free and let $\sigma_1,\ldots ,\sigma_n$ denote a basis of logarithmic $1$-forms and  $\xi_1,\ldots,\xi_n$ the dual logarithmic vector fields.

Let $\pi':X'\rightarrow X$ be the blow-up of $X$ in $p$,
and let $D'$ be the sum of the exceptional divisor~$E_p$ and the strict transform of $D$. 
Since each vector field $\xi_j$ fixes the the point $p$, these vector fields lift to logarithmic vector fields $\xi_1',\ldots,\xi_n'$ of the pair $(X',D')$, which can be proven by the same argument as used for Proposition~\ref{prop: log resolution}.

Let $\tilde\pi:\tilde X\rightarrow X'$ be the functorial log resolution of the pair $(X',D')$, 
and let $\tilde\xi_1,\ldots,\tilde\xi_n$ denote the lifts of 
$\xi_1',\ldots, \xi_n'$ to $\tilde X$.
The composition $\pi=\pi'\circ \tilde\pi:\tilde X\rightarrow 
X$ is also a log resolution of $(X,D)$,
and thus $\sigma_1,\ldots,\sigma_n$ extend to logarithmic $1$-forms $\tilde\sigma_1,\ldots,\tilde\sigma_n$ on $\tilde X$.

Let $E$ denote the exceptional divisor of $\pi$, $\overline{D}$ the strict transform of $D$, and set $\tilde D=E+\overline{D}$.
Since $\pi^{-1}(p)=\tilde\pi^{-1}(E_p)$, the fibre $\pi^{-1}(p)$ has pure codimension $1$, and each irreducible component of $\pi^{-1}(p)$ is a component of the exceptional divisor $E$.
Furthermore, the flows of $\tilde\xi_1,\ldots,\tilde\xi_n$ all stabilise $\pi^{-1}(p)$ since $\xi_1,\ldots,\xi_n$ vanish 
at~$p$
and hence $\tilde\xi_1,\ldots,\tilde\xi_n$ induce vector fields on $\pi^{-1}(p)$.
Let $E_1$ be an irreducible component of $\pi^{-1}(p)$ and let
$q\in E_1$ a point which is not contained in any other irreducible component of $E$ and also not contained in 
$\lfloor \tilde D\rfloor$.
 Since $E_1\subseteq \pi^{-1}(p)$ is projective,
 we may assume that the residues of $\tilde{\sigma}_1,\ldots,
 \tilde{\sigma}_n$ satisfy 
 $$\mathrm{res}_{E_1}(\tilde\sigma_1)=1$$ and 
 $$\mathrm{res}_{E_1}(\tilde\sigma_j)=0$$ if $j>1$.
 
 Therefore $\tilde\sigma_2,\ldots,\tilde \sigma_n$
  induce logarithmic $1$-forms on $E_1$ with respect to the divisor $B=(E_{2}+\ldots +E_k+\overline{D})|_{E_1}$
 if $E_1,\ldots, E_k$ denote the irreducible components of $E$. The restrictions of $\tilde{\xi}_{2},\ldots,\tilde{\xi}_n$ to $E_1$ are dual to these forms.
 Therefore, the sheaves $\Omega_{E_1}^{[1]}(\log \lfloor B\rfloor)$ and $\mathcal T_{E_1}(-\log \lfloor B\rfloor)$ are free.
 By \cite[Corollary~1]{WinkelmannTrivialLogBundle} there is a semi-abelian variety $T$ acting on~$E_1$ with $E_1\setminus\lfloor B\rfloor$ as an open orbit,
 where $T$ admits a short exact sequence of algebraic groups $$0\rightarrow (\mathbb C^*)^{d}\rightarrow T\rightarrow \mathrm{Alb}(E_1)\rightarrow 0$$
 for some $d$ and where $\mathrm{Alb}(E_1)$ denotes the Albanese variety of $E_1$.
 Furthermore, the Lie algebra of $T$ is spanned by the vector fields $\tilde{\xi}_{2}|_{E_1},\ldots,\tilde{\xi}_n|_{E_1}$.

 The flow of each vector field $\tilde{\xi_j}|_{E_1}$ is global, i.e.\ we can take $\mathbb C\times E_1$ as its domain of definition, 
 and for every relatively compact open subset $U\subset \mathbb C$, there is a neighbourhood $V\subset \tilde X$ of $E_1$ such that the flow $\varphi^j$ of $\tilde{\xi}_j$ is defined on 
 $U\times V$, $\varphi^j:U\times V\rightarrow \tilde X$, $(t,x)\mapsto \varphi^j(t,x)=\varphi^j_t(x)$.
 
 Let $\mathcal L=\mathcal O(-E_1)$ denote the line bundle associated with the divisor $E_1$. 
 We have $\varphi^j_t(E_1)=E_1$ and thus get 
 $(\varphi^j_t)^*(\mathcal L|_W)=\mathcal L|_V$ for any $t\in \mathbb C$ and appropriate neighbourhoods $V,W$ of $E_1$ in $\tilde X$ with 
 $\varphi^j_t:V\rightarrow W$.
 Consequently, we have $(\varphi^j_t)^*(\mathcal L|_{E_1})=\mathcal L|_{E_1}$ for all $t\in \mathbb C$ 
 and hence 
 $g^*(\mathcal L|_{E_1})=\mathcal L|_{E_1}$ for any $g\in T$.
 
 By a version of the Negativity Lemma as in \cite[Proposition~1.6]{Graf}, the line bundle 
 $\mathcal L|_{E_1}$ is big.
 Therefore, the stabiliser $\mathrm{St}(\mathcal L|_{E_1})=\{t\in T\,|\, t^*(\mathcal L|_{E_1})=\mathcal L|_{E_1}\}$ of 
 $\mathcal L|_{E_1}$ is contained in the maximal connected linear subgroup $(\mathbb C^*)^d$ of $T$ by \cite[Proposition~5.5.28]{NoguchiWinkelmann}.
 Thus we have $T=(\mathbb C^*)^d$ with $d=n-1$, $\mathrm{Alb}(E_1)=0$, and $E_1$ is a toric variety.  
\end{proof}
\begin{proof}[Proof of (2)]
Let $E_1$ be any irreducible component of $\pi^{-1}(p)$.
 Since $E_1$ is smooth and projective, the action of the torus $T=(\mathbb C^*)^{n-1}$ on $E_1$ has a fixed point $q\in E_1$.
 The Lie algebra of $T$ is spanned by $\tilde{\xi}_2|_{E_1},\ldots,\tilde{\xi}_n|_{E_1}$ and thus we have 
 $\tilde{\xi}_j(q)=0$ for all $j>1$. By construction only~$\tilde{\sigma}_1$ has a pole along~$E_1$ and therefore we necessarily have $\tilde{\xi}_1|_{E_1}=0$ for the dual vector field.
 Moreover, $\tilde{\xi}_1,\ldots,\tilde{\xi}_n$ span the sheaf $\mathcal T_{\tilde X}(-\log\lfloor \tilde D\rfloor)$ on some neighbourhood of $E_1$ and the statement follows.
\end{proof}

\begin{prop}\label{prop: isolated lc case}
  Let $(X,D)$ be an lc pair whose 
  logarithmic tangent sheaf $\mathcal T_X(-\log \lfloor D\rfloor)$ is locally free.
Suppose that there is $p\in X$ such 
that $\xi(p)=0$ for all logarithmic vector fields~$\xi$ defined on some neighbourhood of $p$.

 Then there exist closed logarithmic $1$-forms $\sigma_1,\ldots,\sigma_n$ which span the sheaf $\Omega_X^{[1]}(\log \lfloor D\rfloor)$ in 
 a neighbourhood of $p$.
In particular, the pair $(X,D)$ is toroidal in a neighbourhood of $p$ by Theorem~\ref{thm: closed one forms}.
\end{prop}

The proof consists of two main steps. 
First, we consider a local basis of logarithmic vector fields $\xi_1,\ldots, \xi_n$, and consider their lifts $\tilde\xi_1,\ldots, \tilde\xi_n$ to a log resolution $(\tilde X,\tilde D)$.
The statement of the preceding lemma is then used to study their behaviour near a point $q$ where $\tilde\xi_1(q)=\ldots = \tilde\xi_n(q)=0$
and a version of Poincar\'e's theorem (see e.g.\ \cite[p.\ 190]{Arnold}) on the normal form of holomorphic vector fields allows us to modify
$\xi_1,\ldots, \xi_n$ in such a way that these vector fields are induced by local $\mathbb C^*$-actions, or equivalently by $S^1$-actions, on a neighbourhood of~$p$.

Then, averaging by an appropriate $S^1$-action yields commuting vector fields $\eta_1,\ldots,\eta_n$, which can be shown to still span 
the sheaf of logarithmic vector fields locally near $p$. The logarithmic $1$-forms dual to $\eta_1,\ldots,\eta_n$ are then closed by Lemma~\ref{lemma: closed equals commuting} and yield the desired local basis 
of closed logarithmic $1$-forms.

\begin{proof}
 Let $\pi:\tilde X\rightarrow X$ be a log resolution of the pair $(X,D)$ as in Lemma~\ref{lemma: exceptional divisor is toric}.
 As before, let~$E$ denote the exceptional divisor and $\overline{D}$ the strict transform of $D$, $\tilde D=E+\overline{D}$.
 
 Since the question is local we may assume
 again that $T_X(-\log \lfloor D\rfloor )$ and 
 $\Omega_X^{[1]}(\log\lfloor D\rfloor)$ are free.
 Let $\xi_1,\ldots,\xi_n$ be a basis of logarithmic vector fields and let $\tilde{\xi}_1,\ldots,\tilde{\xi}_n$ denote 
 their lifts to $\tilde X$.
 By Lemma~\ref{lemma: exceptional divisor is toric}, there is a point $q\in \pi^{-1}(p)$ with
 $\tilde{\xi}_1(q)=\ldots =\tilde{\xi}_n(q)=0$.
 Since $\tilde{\xi}_1,\ldots,\tilde{\xi}_n$ form a basis for $\mathcal T_{\tilde X}(-\log \lfloor\tilde D\rfloor)$, 
 $n$ irreducible components $\tilde D_{i_1},\ldots ,\tilde D_{i_n}$ of $\lfloor\tilde D\rfloor$ have to meet in~$q$.
 There exist local coordinates $z_1,\ldots, z_n$ near $q$ such that $q=0$ and locally
 $D_{i_l}=\{z_l=0\}$ for $l=1,\ldots, n$,
 and there are local holomorphic functions 
 $a_{kl}(z)$ 
 such that locally 
 $$\left(\begin{smallmatrix}
         \tilde{\xi}_1\\
         \vdots \\
         \\
         \tilde{\xi}_n
        \end{smallmatrix}\right)
=A(z)\left(\begin{smallmatrix}
            z_1\frac{\partial}{\partial z_1}\\
            \vdots\\
            \\
            z_n\frac{\partial}{\partial z_n}
           \end{smallmatrix}\right)$$
for $A(z)=(a_{kl}(z))_{1\leq k,l\leq n}$.

We now want to prescribe the linear part of the vector field $\tilde{\xi}_1$ at the point $q$ such 
that~$\xi_1$ is conjugated to its linear part and its flow induces a local $\mathbb C^*$-action.
For this, we substitute $\xi_1,\ldots,\xi_n$ by a invertible linear combination of them such 
that $A(0)$ is of the form
$$A(0)=     \left(
     \begin{array}{c|ccc}
       n+1 & n+2 & \cdots & 2n\\ \hline 
       0&  & &  \\  
       \vdots & &\mbox{\LARGE $A_0$ } & \\
       0&  & & 
     \end{array}
     \right)$$          
      
      where $A_0$ is an invertible $(n-1)\times (n-1)$-matrix.

The linear part of $\tilde{\xi}_1$ at $q$ is then given by 
$$(n+1)z_1\frac{\partial}{\partial z_1}+\cdots +2nz_n\frac{\partial}{\partial z_n},$$
and the important point is that its $n$-tuple of eigenvalues $(n+1,\ldots, 2n)$ is non-resonant (in the sense of \cite[$\S$ 22]{Arnold}) and 
the convex hull of the eigenvalues does not contain~$0$.
Hence we may apply Poincar\'e's theorem (see e.g.\ \cite[p.\ 190]{Arnold}) and get that 
there is a neighbourhood of $q$ on which $\tilde{\xi}_1$ is biholomorphically conjugated to its linear part $(n+1)z_1\frac{\partial}{\partial z_1}+\cdots +2nz_n\frac{\partial}{\partial z_n}.$
Moreover, the eigenvalues are all different, which we will need later on.

In a neighbourhood of $q$ the flow $\tilde\varphi^1$ of $\tilde{\xi}_1$ is given by  
$$\left(t,\left(\begin{smallmatrix}
                 w_1\\
                 \vdots\\
                 \\
                 w_n
                \end{smallmatrix}\right)\right)
                \mapsto 
            \left(\begin{smallmatrix}
                 e^{(n+1)t}w_1\\
                 \vdots\\
                 \\
               	 e^{2nt}w_n
                \end{smallmatrix}\right)$$
in appropriate local coordinates $w_1,\ldots, w_n$.
Since $\pi^{-1}(p)$ is compact, the vector field $\tilde{\xi}_1$ induces a global flow on each irreducible component of $\pi^{-1}(p)$,
and there is an open connected neighbourhood $V\subseteq X$ of $p$ such that
$\tilde\varphi^1$ can be defined on $U\times \pi^{-1}(V)$, where $$U=\{t\in\mathbb C\,|\, |\mathrm{Re}(t)|<1,\, |\mathrm{Im}(t)|<4\pi\},$$ and the flow 
$\varphi^1$ of $\xi_1$ is defined on $U\times V$, and we have a commutative diagram:
$$
\xymatrix{
 U\times \pi^{-1}(V)\ar[d]_{\mathrm{id}\times \pi}\ar[rr]^/.8em/{\tilde\varphi^1}& & \tilde X\ar[d]^\pi\\
 U\times V\ar[rr]^/.8em/{\varphi^1} && X
}$$
Locally near $q$ we have $\tilde\varphi^1(2\pi i,w)=w$
and by the identity principle we thus get 
$\tilde\varphi^1(2\pi i,y)=y$ for any $y\in \pi^{-1}(V)$
and moreover $\varphi^1(2\pi i,x)=x$ for any $x\in V$.
Hence the flow map $\varphi^1$ induces a local $\mathbb C^*$-action and 
we may define an $S^1$-action $\chi^1: S^1\times V\rightarrow V$ on $V$ (after possibly shrinking~$V$) by 
setting $\chi^1(e^{is},x)=\varphi^1(is,x)$ (as explained in the proof of Lemma~\ref{lemma: geometry of globalisation})
which induces $\xi_1$.
Moreover, $\tilde{\varphi}^1$ gives rise to an $S^1$-action
$\tilde\chi^1:S^1\times \pi^{-1}(V)\rightarrow \pi^{-1}(V)$ 
which induces the vector field $\tilde\xi_1$.

We now want to use the $S^1$-action $\chi^1:U\times V\rightarrow V$ to average the other vector fields $\xi_j$ and obtain commuting vector fields.
For this purpose we define vector fields
$$\xi_j'=\int_{S^1}(\chi^1_s)_*(\xi_j)\, d\mu(s)$$
for $j\geq 2$, where $\mu$ denotes the unique normalised Haar measure on $S^1$, 
we write $\chi^1_s$ for $\chi^1(s, \cdot)$, and the push-forward
$(\chi^1_s)_*(\xi_j)$ of the vector field $\xi_j$ is as usually defined by 
$$(\chi^1_s)_*(\xi_j)(f)(x)=\xi_j(f\circ \chi^1_s)(\chi^1_{s^{-1}}(x))$$ for any $x\in V$ and local holomorphic function $f$.
The vector fields $\xi_j'$ are all logarithmic with respect to $D$ since the $S^1$-action $\chi^1$ stabilises each irreducible component~$D_i$ of $\lfloor D\rfloor$.
Moreover, for any $t\in S^1$ we have
$$(\chi^1_t)_*(\xi_j')=(\chi^1_t)_*
\int_{S^1}(\chi^1_s)_*(\xi_j)\, d\mu(s)
=\int_{S^1}(\chi^1_t)_*(\chi^1_s)_*(\xi_j)\, d\mu(s)
=\int_{S^1}(\chi^1_{st})_*(\xi_j)\, d\mu(s)=\xi_j'$$
due to the invariance of the Haar measure.
This implies 
$(\varphi^1_t)_*(\xi_j')=\xi_j'$ for
any $t\in \mathbb C$ in a neighbourhood of~$0$.
Consequently, 
the vector fields $\xi_1$ and $\xi_j'$ commute:
$$[\xi_1,\xi_j']=
-\left.\frac{d}{dt}\right|_{t=0}(\varphi^1_t)_*(\xi_j')=
-\left.\frac{d}{dt}\right|_{t=0}\xi_j'=0$$

Next, we prove that $\xi_1',\ldots ,\xi_n'$ (where we set $\xi_1'=\xi_1$) still form a local basis for the logarithmic tangent sheaf $\mathcal T_X(-\log \lfloor D\rfloor)$ near $p$ and that $\xi_1',\ldots ,\xi_n'$  are pairwise commuting.
In order to do so, we consider the lifts of $\xi_j'$ to $\pi^{-1}(V)$, which are given by 
$$\tilde\xi_j'=\int_{S^1}(\tilde\chi_s^1)_*(\tilde\xi_j)\,d\mu(s),$$
and analyse them near the point $q$.
Recall that in appropriate coordinates $w_1,\ldots, w_n$ with $w_j(q)=0$
we have $\tilde\xi_1'=\tilde\xi_1=(n+1)w_1\frac{\partial}{\partial w_1}+\ldots
+2nw_n\frac{\partial}{\partial w_n}$ near $q$.
Let $b_{jkl}(w)$ be holomorphic functions defined locally near $q$
such that
$$\tilde\xi_j=\sum_{k,l=1}^n b_{jkl}(w)w_k\frac{\partial}{\partial w_l}$$
for $j\geq 2$.
Then a calculation in local coordinates yields that
$\tilde\xi_j'$ is linear (with respect to the coordinates $w_1,\ldots, w_n$):
$$\tilde\xi_j'=\sum_{k,l=1}^n b_{jkl}(0)w_k\frac{\partial}{\partial w_l}$$
Moreover, since all eigenvalues of $\tilde\xi_1=\tilde\xi_1'$ at $q=0$ are different and 
$\tilde\xi_j'$ and $\tilde\xi_1'$ commute we get
that 
$b_{jkl}(0)=0$ if $k\neq l$ and hence $\tilde\xi_j'$ is of the form
$$\tilde\xi_j'=\sum_{k=1}^n b_{jk}w_k\frac{\partial}{\partial w_k}$$
for some constants $b_{jk}$.
In particular, we see now that 
the vector fields $\tilde\xi_j'$ are all pairwise commuting near $q$, thus by the identity principle on all of $\pi^{-1}(V)$ and 
consequently $\xi_1',\ldots,\xi_n'$ are also pairwise commuting vector fields.

Moreover, we have 
$$\left(\begin{smallmatrix}
\tilde\xi_1'\\
\vdots\\
\\
\tilde\xi_n'
\end{smallmatrix}\right)
=\tilde{C}(w) 
\left(\begin{smallmatrix}
\tilde\xi_1\\
\vdots\\
\\
\tilde\xi_n
\end{smallmatrix}\right)
$$
for some matrix $\tilde C(w)$ whose entries $\tilde c_{jk}(w)$ are local holomorphic functions and 
which satisfies $\tilde C(q)=\tilde C(0)=E_n$.

Since $\xi_1,\ldots, \xi_n$ form a basis of logarithmic vector fields on $X$, we have 
$$\left(\begin{smallmatrix} 
\xi_1'\\
\vdots\\
\\
\xi_n'
\end{smallmatrix}\right)
={C}(x) 
\left(\begin{smallmatrix}
\xi_1\\
\vdots\\
\\
\xi_n
\end{smallmatrix}\right)
$$
for a matrix $C(x)$ whose entries $c_{jk}(x)$ are holomorphic functions on a neighbourhood of~$p$.
Using that $\tilde\xi_1,\ldots,\tilde\xi_n$ are the lifts of 
$\xi_1,\ldots,\xi_n$ and $\tilde\xi_1',\ldots, \tilde\xi_n'$ the lifts
of $\xi_1',\ldots,\xi_n'$, we get
$C(\pi(y))=\tilde C(y)$ on a neighbourhood of $q\in \tilde X$ 
and in particular
$C(p)=C(\pi(q))=\tilde C(q)=E_n$ is invertible.
Hence $\xi_1',\ldots ,\xi_n'$ also form a local basis of logarithmic vector fields on $X$ near~$q$.
These vector fields $\xi_1',\ldots,\xi_n'$ commute 
and their dual logarithmic $1$-forms $\sigma_1,\ldots,\sigma_n$ are thus closed (cf.\ Lemma~\ref{lemma: closed equals commuting}).
\end{proof}

The statement of the next lemma on hyperplane sections will be useful when
reducing the case of an lc pair $(X,D)$ with locally free sheaf
$\Omega_X^{[1]}(\log \lfloor D\rfloor)$ to the isolated case as in Proposition~\ref{prop: isolated lc case}.
For more details on hyperplane sections and their properties relevant to our setting the reader is referred to \cite[Section~2.E]{GKKP11}.

\begin{lemma}\label{lemma: hyperplanes and locally free}
Let $(X,D)$ be an lc pair, $D_1,\ldots D_k$ the irreducible components of $D$, $D=\sum_i a_i D_i$, and let 
$H$ be a general member of an ample basepoint free linear system on~$X$.

If the sheaf $\Omega_X^{[1]}(\log \lfloor D\rfloor)$ is locally free,
then $\Omega_H^{[1]}(\log \lfloor D\rfloor|_H)$ is locally free.
\end{lemma}

\begin{rmk}\label{rmk: hyperplane section}
By \cite[Lemma~2.23]{GKKP11} the divisor $H$ is normal and irreducible, and 
the intersections $D_j\cap H$ are all distinct.
Therefore, $(H,D|_H)$ with $D|_H=a_1(D_1\cap H)+\ldots +a_k(D_k\cap H)$
is a pair, and $(H, D|_H)$ is lc if $(X,D)$ is lc; see \cite[Lemma~2.25]{GKKP11}.
\end{rmk}

\begin{proof}[Proof of Lemma~\ref{lemma: hyperplanes and locally free}]
Since the question is local, we may assume that 
$\Omega_X^{[1]}(\log \lfloor D\rfloor)$ is free and~$H$ is given by the reduced equation $h=0$ for a regular function $h$ on $X$.

Let $\pi:\tilde X\rightarrow X$ be a log resolution of $(X,D)$ and let $\tilde H=\pi^{-1}(H)$.
By \cite[Lemma~2.24]{GKKP11} the restricted morphism
$\pi|_{\tilde H}:\tilde H\rightarrow H$ is a log resolution of
the pair $(H, D|_H)$, 
and the exceptional sets $\mathrm{Exc}(\pi)$ of $\pi$ and 
$\mathrm{Exc}(\pi|_{\tilde H})$ of $\pi|_{\tilde H}$
satisfy $\mathrm{Exc}(\pi|_{\tilde H})=\mathrm{Exc}(\pi)\cap \tilde H$.

Let $\sigma_1,\ldots,\sigma_n$ be a basis of logarithmic $1$-forms on $(X,D)$, and let $\tilde \sigma_1,\ldots,\tilde\sigma_n$
denote their lifts to $\tilde X$.
The hyperplane $\tilde H\subset \tilde X$ is given by the reduced equation $\tilde h=h\circ \pi=0$.
Let $\alpha_1,\ldots,\alpha_n$ be regular functions on $X$
such that 
$$dh=\sum_{j=1}^n\alpha_j\sigma_j\ \ \ \text{ and }\ \ \ 
d\tilde h=\sum_{j=1}^n (\alpha_j\circ\pi)\tilde\sigma_j.$$
After possibly shrinking $X$, the $1$-form $d\tilde h$ has no zeroes and there is $j$, say $j=1$, 
with $\alpha_j(\pi(y))=\alpha_1(\pi(y))\neq 0$ for all $y\in \tilde X$.
Consequently, we may assume $d\tilde h=\tilde \sigma_1$ and $dh=\sigma_1$ without loss of generality.

Let $H^\circ$ be largest open subset of $H$ such that
$(H^\circ, D|_{H^\circ})$ is snc.
Then $(X,D)$ is snc along~$H^\circ$
and the restrictions of $\sigma_2,\ldots,\sigma_n$ to
logarithmic forms $\sigma_2|_H,\ldots,\sigma_n|_H$ in 
$\Omega_H^{[1]}(\log \lfloor D\rfloor|_H)$ are well-defined.
On $H^0$ we have an exact sequence 
$$0\rightarrow \mathcal O_{H^\circ} \langle h\rangle \rightarrow 
\Omega_X^{[1]}(\log \lfloor D\rfloor)|_{H^\circ}\rightarrow \Omega_H^{[1]}(\log \lfloor D\rfloor|_{H})|_{H^\circ}\rightarrow 0,$$
where the kernel of the morphism 
$\Omega_X^{[1]}(\log \lfloor D\rfloor)|_{H^\circ}\rightarrow \Omega_H^{[1]}(\log \lfloor D\rfloor|_{H})|_{H^\circ}$ is generated by 
$dh=\sigma_1$.
Therefore, $\sigma_2|_{H^\circ},\ldots,\sigma_n|_{H^\circ}$ are a basis for $\Omega_H^{[1]}(\log \lfloor D\rfloor|_{H})|_{H^\circ}$
and thus $\sigma_2|_H,\ldots,\sigma_n|_H$ form a basis of 
$\Omega_H^{[1]}(\log \lfloor D\rfloor|_{H})$ since 
$H\setminus H^\circ$ has at least codimension~2.
Hence $\Omega_H^{[1]}(\log \lfloor D\rfloor|_{H})$ is locally free.
\end{proof}

\begin{thm}\label{thm: locally free lc pair}
Let $(X,D)$ be an lc pair such that $\Omega_X^{[1]}(\log\lfloor D\rfloor)$ is locally free.
Then $(X,\lfloor D\rfloor)$ is toroidal.
\end{thm}

\begin{proof}
Let $Z\subset X$ be the smallest closed analytic subset such 
that the pair $(X\setminus Z, \lfloor D\rfloor|_{X\setminus Z})$ is toroidal. 
We shrink $X$ such that $\mathcal T_X(-\log \lfloor D\rfloor )$ and
$\Omega_X^{[1]}(\log \lfloor D\rfloor )$ are free and $Z$ is connected.
We now want to do induction on the dimension of~$Z$.

If $Z$ is $0$-dimensional, then $Z$ consists of a single point $Z=\{p\}$, and the statement of the theorem is the content of Proposition~\ref{prop: isolated lc case}.
Assume now that $m=\dim Z$ and that the theorem is proven for those pairs $(X',D')$ such that the non-toroidal locus $Z'$ of the pair $(X',\lfloor D'\rfloor)$ has
$\dim Z'<m$.
Let $H$ be a general hyperplane section of an ample basepoint free linear system on $X$ as described in Lemma~\ref{lemma: hyperplanes and locally free}.
Then $(H,D|_H)$ is lc, $\Omega_H^{[1]}(\log \lfloor D|_H\rfloor)$ is locally free
and 
$\dim(Z\cap H)=
\dim Z-1=m-1<m$. Hence $(H,\lfloor D|_H\rfloor)$ is toroidal by the induction hypothesis.

Let $\pi:\tilde X\rightarrow X$ be the functorial log resolution of $(X,D)$ as in Proposition~\ref{prop: log resolution},
$\overline{D}$ the strict transform of $D$, $E$ the exceptional divisor, $\tilde D=E+\overline{D}$.
Let $\xi_1,\ldots,\xi_n$ be a basis of logarithmic vector fields and let $\tilde\xi_1,\ldots,\tilde \xi_n$ be their lifts to $\tilde X$.
Since $\tilde\xi_1,\ldots,\tilde \xi_n$ are a basis for
$\mathcal T_{\tilde X}(-\log \lfloor \tilde D\rfloor)$ and their
flows all stabilise $\pi^{-1}(Z)$, the flows of $\tilde\xi_1,\ldots,\tilde \xi_n$ act transitively on a neighbourhood $V\subset \pi^{-1}(Z)$ of $\pi^{-1}(p)\in\pi^{-1}(Z)$ 
for a general point $p\in Z$.
We get that 
the flows of $\xi_1,\ldots,\xi_n$ act transitively on the neighbourhood $\pi(V)$ of $p$ in $Z$.

Let $p\in Z$ such that $p\in H\cap Z$
and let $\xi$ be a vector fields on a neighbourhood of $p$ which is not 
tangent to $H$ at $p$.
We may assume that $X\subset \mathbb A^n$ and $H$ is the intersection of a smooth divisor $\hat{H}$ and $X$.
The vector field $\xi$ extends to a holomorphic vector field $\hat{\xi}$ on  an open neighbourhood of $p\in X\subset \mathbb A^n$ in $\mathbb A^n$.
Let $\hat\varphi:\Omega\rightarrow \mathbb A^n$, $\Omega\subseteq \mathbb C\times\mathbb A^n$, denote the flow map of $\hat{\xi}$.
Since $\xi$ is not tangent to $H$ at $p$, $\hat{\xi}$ is not tangent to $\hat{H}$ at $p$ and the flow $\hat\varphi$
induces a morphism $\chi:U\times \hat{H}\rightarrow 
\mathbb A^n$, $(t,q)\mapsto \hat\varphi(t,q)$, where $U$ is an open subset of $\mathbb C$ with $0\in U$,  such that $\chi$ is biholomorphic near $p$. 
Moreover, we have $\chi(U\times H)\subseteq X$ by construction, and thus we get that $U\times H$ and $X$ are biholomorphic near~$p$. In particular, it follows that~$X$ is toroidal in a neighbourhood of $p$,
which is a contradiction to our assumption that $p$ is contained in the non-toroidal subset $Z$ of $X$.
\end{proof}

A version of Theorem~\ref{thm: locally free lc pair} for Du Bois pairs can now directly be deduced by applying the results of 
\cite{GrafKovacsDuBois}.
For definitions and a detailed discussion of Du Bois pairs the reader is referred to \cite[Chapter 6]{KollarSingularitiesMMP}.

\begin{cor}\label{cor: DuBois case}
Let $X$ be a normal quasi-projective variety and $\Sigma\subsetneq X$ a reduced closed subscheme
such that $(X,\Sigma)$ is a Du Bois pair.
Let $\Sigma_{\mathrm{div}}$ denote the largest reduced divisor whose support is contained in $\Sigma$.
Assume that 
$\Omega_X^{[1]}(\log \Sigma_{\mathrm{div}})$ is locally free.
Then $(X, \Sigma_{\mathrm{div}})$ is toroidal.
\end{cor}

\begin{proof}
It is again enough to consider the case where $\mathcal T_X(-\log \Sigma_{\mathrm{div}})$ and
$\Omega_X^{[1]}(\log  \Sigma_{\mathrm{div}})$ are free.
Then the twisted canonical sheaf $\omega_X( \Sigma_{\mathrm{div}})\cong\Omega_X^{[n]}(\log  \Sigma_{\mathrm{div}})$, $n=\dim X$,
is also free, which implies that the divisor 
$K_X+ \Sigma_{\mathrm{div}}$ is linearly equivalent to $0$, 
where $K_X$ denote a canonical divisor of $X$.
In particular, $K_X + \Sigma_{\mathrm{div}}$ is Cartier.
Therefore, the pair $(X, \Sigma_{\mathrm{div}})$ is lc by
\cite[Theorem 1.4.2]{GrafKovacsDuBois}
and $(X, \Sigma_{\mathrm{div}})$ is toroidal by Theorem~\ref{thm: locally free lc pair}.
\end{proof}

\begin{rmk}
Alternatively, the statement of Corollary~\ref{cor: DuBois case} could be proven along the lines as the statement for lc pairs noting that extension of logarithmic forms to log resolutions and 
the cutting down procedure via hyperplanes also work for Du Bois pairs by 
\cite[Theorem~4.1 and Lemma 4.4]{GrafKovacsDuBois}
\end{rmk}

\bibliographystyle{amsalpha}

\begin{thebibliography}{GKKP11}

\bibitem[Arn88]{Arnold}
Vladimir~I. Arnol'd, \emph{Geometrical methods in the theory of ordinary
  differential equations}, second ed., Grundlehren der Mathematischen
  Wissenschaften [Fundamental Principles of Mathematical Sciences], vol. 250,
  Springer-Verlag, New York, 1988, Translated from the Russian by Joseph
  Sz\"ucs [J\'ozsef M. Sz\H ucs]. \MR{947141}

\bibitem[Dru14]{Druel}
St\'ephane Druel, \emph{The {Z}ariski-{L}ipman conjecture for log canonical
  spaces}, Bull. Lond. Math. Soc. \textbf{46} (2014), no.~4, 827--835.
  \MR{3239620}

\bibitem[Fab15]{Faber}
Eleonore Faber, \emph{Characterizing normal crossing hypersurfaces}, Math. Ann.
  \textbf{361} (2015), no.~3-4, 995--1020. \MR{3319556}

\bibitem[Fis76]{Fischer}
Gerd Fischer, \emph{Complex analytic geometry}, Lecture Notes in Mathematics,
  Vol. 538, Springer-Verlag, Berlin-New York, 1976. \MR{0430286}

\bibitem[GK14a]{GrafKovacs}
Patrick Graf and S\'andor~J. Kov\'acs, \emph{An optimal extension theorem for
  1-forms and the {L}ipman-{Z}ariski conjecture}, Doc. Math. \textbf{19}
  (2014), 815--830. \MR{3247804}

\bibitem[GK14b]{GrafKovacsDuBois}
\bysame, \emph{Potentially {D}u {B}ois spaces}, J. Singul. \textbf{8} (2014),
  117--134. \MR{3395242}

\bibitem[GKK10]{GKK10}
Daniel Greb, Stefan Kebekus, and S\'andor~J. Kov\'acs, \emph{Extension theorems
  for differential forms and {B}ogomolov-{S}ommese vanishing on log canonical
  varieties}, Compos. Math. \textbf{146} (2010), no.~1, 193--219. \MR{2581247}

\bibitem[GKKP11]{GKKP11}
Daniel Greb, Stefan Kebekus, S{\'a}ndor~J. Kov{\'a}cs, and Thomas Peternell,
  \emph{Differential forms on log canonical spaces}, Publ. Math. Inst. Hautes
  \'Etudes Sci. (2011), no.~114, 87--169. \MR{2854859}

\bibitem[Gra15]{Graf}
Patrick Graf, \emph{Bogomolov-{S}ommese vanishing on log canonical pairs}, J.
  Reine Angew. Math. \textbf{702} (2015), 109--142. \MR{3341468}

\bibitem[GS14]{GrangerSchulze}
Michel Granger and Mathias Schulze, \emph{Normal crossing properties of complex
  hypersurfaces via logarithmic residues}, Compos. Math. \textbf{150} (2014),
  no.~9, 1607--1622. \MR{3260143}

\bibitem[Hei91]{HeinznerInvariantTheory}
Peter Heinzner, \emph{Geometric invariant theory on {S}tein spaces}, Math. Ann.
  \textbf{289} (1991), no.~4, 631--662. \MR{1103041}

\bibitem[J{\"o}r14]{JoerderWLZ}
Clemens J{\"o}rder, \emph{A weak version of the {L}ipman-{Z}ariski conjecture},
  Math. Z. \textbf{278} (2014), no.~3-4, 893--899. \MR{3278896}

\bibitem[Kau65]{Kaup}
Wilhelm Kaup, \emph{Infinitesimale {T}ransformationsgruppen komplexer
  {R}\"aume}, Math. Ann. \textbf{160} (1965), 72--92. \MR{0181761}

\bibitem[KM98]{KollarMori}
J\'anos Koll\'ar and Shigefumi Mori, \emph{Birational geometry of algebraic
  varieties}, Cambridge Tracts in Mathematics, vol. 134, Cambridge University
  Press, Cambridge, 1998, With the collaboration of C. H. Clemens and A. Corti,
  Translated from the 1998 Japanese original. \MR{1658959}

\bibitem[Kol97]{KollarSingOfPairs}
J\'anos Koll\'ar, \emph{Singularities of pairs}, Algebraic geometry---{S}anta
  {C}ruz 1995, Proc. Sympos. Pure Math., vol.~62, Amer. Math. Soc., Providence,
  RI, 1997, pp.~221--287. \MR{1492525}

\bibitem[Kol07]{KollarSingularities}
\bysame, \emph{Lectures on resolution of singularities}, Annals of Mathematics
  Studies, vol. 166, Princeton University Press, Princeton, NJ, 2007.
  \MR{2289519}

\bibitem[Kol13]{KollarSingularitiesMMP}
\bysame, \emph{Singularities of the minimal model program}, Cambridge Tracts in
  Mathematics, vol. 200, Cambridge University Press, Cambridge, 2013, With a
  collaboration of S\'andor Kov\'acs. \MR{3057950}

\bibitem[Lip65]{Lipman}
Joseph Lipman, \emph{Free derivation modules on algebraic varieties}, Amer. J.
  Math. \textbf{87} (1965), 874--898. \MR{0186672}

\bibitem[NW14]{NoguchiWinkelmann}
Junjiro Noguchi and J\"org Winkelmann, \emph{Nevanlinna theory in several
  complex variables and {D}iophantine approximation}, Grundlehren der
  Mathematischen Wissenschaften [Fundamental Principles of Mathematical
  Sciences], vol. 350, Springer, Tokyo, 2014. \MR{3156076}

\bibitem[Oda88]{OdaToricGeometry}
Tadao Oda, \emph{Convex bodies and algebraic geometry}, Ergebnisse der
  Mathematik und ihrer Grenzgebiete (3) [Results in Mathematics and Related
  Areas (3)], vol.~15, Springer-Verlag, Berlin, 1988, An introduction to the
  theory of toric varieties, Translated from the Japanese. \MR{922894}

\bibitem[Sai80]{Saito}
Kyoji Saito, \emph{Theory of logarithmic differential forms and logarithmic
  vector fields}, J. Fac. Sci. Univ. Tokyo Sect. IA Math. \textbf{27} (1980),
  no.~2, 265--291. \MR{586450}

\bibitem[Sno82]{Snow}
Dennis~M. Snow, \emph{Reductive group actions on {S}tein spaces}, Math. Ann.
  \textbf{259} (1982), no.~1, 79--97. \MR{656653}

\bibitem[Win04]{WinkelmannTrivialLogBundle}
J\"org Winkelmann, \emph{On manifolds with trivial logarithmic tangent bundle},
  Osaka J. Math. \textbf{41} (2004), no.~2, 473--484. \MR{2069097}

\end{thebibliography}
\providecommand{\bysame}{\leavevmode\hbox to3em{\hrulefill}\thinspace}
\providecommand{\MR}{\relax\ifhmode\unskip\space\fi MR }
\providecommand{\MRhref}[2]{%
  \href{http://www.ams.org/mathscinet-getitem?mr=#1}{#2}
}
\providecommand{\href}[2]{#2}

\end{document}